\documentclass[a4paper,10pt]{article}

\usepackage{amsthm}

\usepackage{lipsum}
\usepackage{amsfonts}
\usepackage{subcaption,graphicx}
\usepackage{epstopdf}
\usepackage{algorithmic}
\usepackage{amsmath,amssymb}
\usepackage{tikz}
\usepackage{hyperref}
\usepackage[noabbrev, capitalise, nameinlink]{cleveref}
\usepackage{multirow}
\usepackage{mathtools}
\usepackage{graphicx}
\usepackage{tikz,pgfplots}
\usepackage[square,numbers]{natbib}
\usepackage{easytable}
\usepackage{tabularray}
\usepackage{dsfont}

\usepackage[verbose,a4paper,tmargin=35mm,bmargin=32mm,lmargin=28mm,rmargin=28mm]{geometry}

\usepackage{color}

\usepackage[utf8]{inputenc}

\usepackage{xfrac}
\usepackage{nicefrac}
\usepackage{mathtools}

\usepackage{array}
\makeatletter
\newcolumntype{"}{@{\hskip\tabcolsep\vrule width 1pt\hskip\tabcolsep}}
\makeatother

\usepackage{framed}
\usepackage{graphicx, caption, subcaption}
\usepackage{float}

\newtheorem{theorem}{Theorem}

\newtheorem{lemma}[theorem]{Lemma}
\theoremstyle{definition}

\theoremstyle{remark}
\newtheorem{remark}[theorem]{Remark}

\Crefname{assumption}{Assumption}{Assumptions}
\numberwithin{theorem}{section}
\numberwithin{equation}{section}
\numberwithin{table}{section}
\numberwithin{figure}{section}

\ifpdf
\DeclareGraphicsExtensions{.eps,.pdf,.png,.jpg}
\else
\DeclareGraphicsExtensions{.eps}
\fi

\newcommand{{\hc}}[1]{\textcolor{blue}{ #1}}

\newcommand{\dx}{\,\mathrm{d}x}
\newcommand{\T}{\mathcal{T}}

\newcommand{\Ab}{\mathbf{A}}

\newcommand{\bfideal}{\varphi}

\DeclareMathOperator*{\argmin}{arg\,min}

\def\hyph{-\penalty0\hskip0pt\relax}

\title{\textbf{Hierarchical Super-Localized Orthogonal Decomposition Method}\footnote{
		The work of the authors is part of a project that has received funding from the European Research Council (ERC) under the European Union's Horizon 2020 research and innovation programme (Grant agreement No.~865751 -- RandomMultiScales).
		
		$\dagger$ Institute of Mathematics, University of Augsburg, Universit\"atsstr.~12a, 86159 Augsburg, Germany
		
		$\star$ Centre for Advanced Analytics and Predictive Sciences (CAAPS), University of Augsburg, Universit\"atsstr.~12a, 86159 Augsburg, Germany}}

\author{Jos\'e C. Garay$^\dagger$ \and Hannah Mohr$^\dagger$\and Daniel Peterseim$^{\dagger,\star}$\and Christoph Zimmer$^\dagger$ }

\date{\vspace{-5ex}}

\begin{document}
	\maketitle
	
\begin{abstract}
We present the construction of a sparse-compressed operator that approximates the solution operator of elliptic PDEs with rough coefficients. To derive the compressed operator, we construct a hierarchical basis of an approximate solution space, with superlocalized basis functions that are {quasi-}orthogonal across hierarchy levels with respect to the inner product induced by the energy norm. The superlocalization is achieved through a novel variant of the Super-Localized Orthogonal Decomposition method that is built upon corrections of basis functions arising from the Localized Orthogonal Decomposition method. The hierarchical basis not only induces a sparse compression of the solution space but also enables an orthogonal multiresolution decomposition of the approximate solution operator, decoupling scales and solution contributions of each level of the hierarchy. With this decomposition, the solution of the PDE reduces to the solution of a set of independent linear systems per level with mesh-independent condition numbers that can be computed simultaneously. We present an accuracy study of the compressed solution operator as well as numerical results illustrating our theoretical findings and beyond, revealing that desired optimal error rates with well-behaved superlocalized basis functions can still be attained even in the challenging case of coefficients with high-contrast channels.
\end{abstract}

\vspace{1cm}
\noindent\textbf{Keywords:}
Hierarchical basis, superlocalization, numerical homogenization,  orthogonal multiresolution decomposition, sparse compression, rough coefficients, multiscale method
\\[2ex]
\textbf{AMS subject classifications:}
65N12,
65N15, 	
65N30

\section{Introduction}

Let $\Omega\in \mathbb{R}^d$ ($d=1,2,3$) be a bounded Lipschitz polytope and $f\in L^2(\Omega)$. We consider the following elliptic problem {encoding linear diffusion type problems}: Find $u\in H^{1}_{0}(\Omega)$ such that
\begin{equation}\label{Continuous-Formulation}
a(u,v)=(f,v)_{L^2(\Omega)}\;\;\;\;\;\;\text{for all } v\in H^{1}_{0}(\Omega),
\end{equation}
where
\begin{equation}
a(u,v)\coloneqq\int_{\Omega}(\mathbf{A}\nabla u) \cdot \nabla v \dx.
\end{equation}
Here, $\mathbf{A}\in L^{\infty}(\Omega,\mathbb{R}^{d\times d})$ is a symmetric matrix-valued function for which there exist constants $0<\alpha\leq\beta<\infty$ such that for almost all $x\in\Omega$ and for all $\eta\in \mathbb{R}^d$
\begin{equation}\label{A-spectral-bound}
\alpha\|\eta\|_2^2\leq(\mathbf{A}(x)\eta)\cdot\eta\leq\beta\|\eta\|_2^2.
\end{equation}
Since the entries of $\mathbf{A}$ are in $L^{\infty}(\Omega)$, the coefficient $\mathbf{A}$ is called {\it rough}. The (linear) solution operator of (\ref{Continuous-Formulation}) is $\mathcal{A}^{-1}:L^2(\Omega)\rightarrow H_{0}^{1}(\Omega)$, $f\mapsto u$. 
\begin{remark}
The method presented in this work can be extended to the case where $\mathbf{A}$ is non-symmetric; see, e.g.,  \cite{HaPe21}. Additionally, without loss of generality, we could assume $\alpha=1$ (see \cite{Brenner2022}).
\end{remark}

Let $\mathcal{T}_h$ be a Cartesian mesh of $\Omega$ with mesh size $h$, and let $V_h$ be the $\mathcal{Q}_{1}$ finite element space associated with $\mathcal{T}_h$ \cite{BrS08}. We assume that $h$ is small enough so that $\mathcal{T}_h$ resolves all characteristic lengths of $\mathbf{A}$ and the linear operator $\mathcal{A}^{-1}_{h}:L^{2}(\Omega)\rightarrow V_h$ given by
\begin{equation}\label{Fine-problem-def}
	a(\mathcal{A}^{-1}_{h}g,v)=(g,v)_{L^{2}(\Omega)} \;\;\;\text{for all }\;v\in V_{h} \;\text{and } g\in L^{2}(\Omega)
\end{equation}
is a good approximation of $\mathcal{A}^{-1}$ in the sense that $\|\mathcal{A}^{-1}f-\mathcal{A}_h^{-1}f\|_a$ is small for all $f\in L^2(\Omega)$. Here $\|\cdot\|_a$ is the (energy) norm induced by $a(\cdot,\cdot)$.
Denote by $\mathcal{T}_H$ the Cartesian mesh of $\Omega$ with mesh size $H>h$, and let $\mathcal{V}_H=\mathbf{span}\{\mathcal{A}^{-1}_h \chi_{T} \;:\; T\in \mathcal{T}_H\}$, where $\chi_{T}$ is the characteristic function of $T$. The discrete approximate-solution operator of problem (\ref{Continuous-Formulation}) in $\mathcal{V}_H$ is defined as $\mathcal{A}^{-1}_{H}:L^2(\Omega)\rightarrow \mathcal{V}_H$ such that
\begin{equation}\label{Coarse-problem-def}
a(\mathcal{A}^{-1}_{H} f,v_H)=(f,v_H)_{L^2(\Omega)}\;\;\;\;\text{for all } v_H\in\mathcal{V}_H.
\end{equation}
From the definition of $\mathcal{V}_{H}$ we observe that $\mathcal{A}^{-1}_H f=\mathcal{A}^{-1}_h f$ if $f\in \mathbb{Q}^{0}(\mathcal{T}_{H})$. In other terms, if $\Pi_H:L^2(\Omega)\rightarrow \mathbb{Q}^{0}(\mathcal{T}_H)$ denotes the $L^2$-orthogonal projection onto $\mathbb{Q}^{0}(\mathcal{T}_H)$, then $\mathcal{A}^{-1}_{H}\circ \Pi_H=\mathcal{A}^{-1}_h\circ\Pi_H$.
This equality implies that $\mathcal{A}^{-1}_H$ should also be a good approximation of $\mathcal{A}^{-1}$ whenever $\|f-\Pi_{H}f\|_{L^2(\Omega)}$ is small \cite{AHP21,MalP20}. The advantage of solving \eqref{Coarse-problem-def} instead of \eqref{Fine-problem-def} is that it requires the solution of a much smaller linear system (in an online stage), with the caveat that the assembly of the associated stiffness matrix requires the prior computation of the basis functions $\mathcal{A}^{-1}_h\chi_{T}$ (in an offline stage). Consequently, since such basis functions are independent of $f$, the use of \eqref{Coarse-problem-def} as a method to solve \eqref{Continuous-Formulation} is justified whenever there is a need to solve \eqref{Continuous-Formulation} multiple times with different $f\in L^2(\Omega)$, provided that $\|f-\Pi_{H}f\|_{L^2(\Omega)}$ is small and the sum of the offline and online times do not exceed the time required to solve \eqref{Fine-problem-def} the same number of times. Unfortunately, the natural basis functions $\mathcal{A}^{-1}_h\chi_{T}$ are slowly-decaying and globally supported (i.e., their values are not negligible outside a local (small) patch). As a consequence, not only are they expensive to compute when multiple scales are present in $\mathbf{A}$ (in particular with a rough $\mathbf{A}$), but also the associated stiffness matrix becomes dense, rendering the Finite Element solution of (\ref{Continuous-Formulation}) impractical with this basis. To obtain a computationally more advantageous solution scheme, it would be ideal to have a basis of $\mathcal{V}_H$ comprised of locally supported functions obtained through local computations. In later sections, we shall see that this is not always possible, but in such cases we can still find basis functions that are rapidly decaying and well approximated by locally computable and supported functions. The space spanned by these localized basis functions will generally not coincide with $\mathcal{V}_H$ but still be a discrete solution space of (\ref{Continuous-Formulation}) with good approximation properties. 

Numerical homogenization methods such as the Localized Orthogonal Decomposition (LOD) \cite{MaP14, KPY18, PeS16, HeP13, BrennerLOD} provide a way to construct the aforementioned approximate solution spaces for (\ref{Continuous-Formulation}) spanned by local basis functions. In this method, local basis functions are obtained by truncating the support of (in most cases) exponentially decaying functions spanning $\mathcal{V}_H$. {Other numerical homogenization approaches that could also be employed to obtain localized basis functions are, e.g., the Multiscale Finite Element method (MsFEM) \cite{HoW97,EH09,blanc2023homogenization}, the Generalized Multiscale Finite Element method (GMsFEM) \cite{EFENDIEV2013116}, the Generalized Finite Element method (GFEM)\cite{BaL11, BLS20}, Bayesian methods \cite{Owh15}, rough polyharmonic splines methods \cite{OZB14}, and FETI-DP and BDDC inspired multiscale methods \cite{MaSa21,KKR18}; we refer to \cite{AHP21,CEH23} for more comprehensive reviews.}

Given a hierarchy of nested meshes $\{\mathcal{T}_{H_{\ell}}\}_{\ell\in\{0,\ldots,L\}}$ (with $H_i<H_j$ for $j<i$) obtained by multiple refinements of a given mesh $\mathcal{T}_{H_0}$, and a set of nested function spaces $\{\mathcal{V}_{H_{\ell}}\}_{\ell\in\{0,\ldots,L\}}$ associated with each of these meshes, a hierarchical basis of $\mathcal{V}_{H_{L}}$ is defined as one that incorporates functions from each of the nested spaces. A solution space of (1.1) can also be constructed using a hierarchical basis with $a$-orthogonal basis functions across levels \cite{Owh17,FeP20}. This scheme is particularly useful in the presence of multiple scales within the diffusion matrix $\mathbf{A}$. The $a$-orthogonality across levels allows for an orthogonal multiresolution decomposition of the solution operator. In other words, the solution operator can be expressed as the sum of certain operators, each of these representing the contribution to the solution of one of the levels of the hierarchy, with contributions that are independent of each other. Moreover, the multiresolution qualifier in this context refers to the decoupling of scales across levels arising from this decomposition. This decoupling implies that each level contains only the information of the scales necessary to achieve an accuracy determined by the mesh size of that level and the function $f$. Consequently, to attain a particular accuracy, we need to consider the aggregate contribution of all the levels up to the level associated with the mesh size corresponding to the required accuracy. An important aspect of this hierarchical scheme is that, given the independence of level contributions, these contributions can be computed simultaneously, reducing the computational time of the aggregate solution. In algebraic terms, the stiffness matrix associated with the $a$-orthogonal hierarchical basis possesses a block-diagonal structure, where each diagonal block corresponds to a particular level of the hierarchy and is obtained from the basis functions belonging to that particular level.

{The construction of a hierarchical basis inducing an orthogonal multiresolution decomposition of the solution operator was first introduced in \cite{Owh17}, where basis functions were denominated {\it gamblets}}. In \cite{FeP20}, a hierarchical basis was obtained connecting the concept of gamblets with the LOD framework, and combining the notions of Haar basis functions and the multiresolution LOD method. In that work, the Haar basis, i.e., an $L^2$-orthogonal hierarchical basis of $\mathbb{Q}^{0}(\mathcal{T}_H)$ with locally supported basis functions is used as a discontinuous companion of a certain regularized hierarchical basis constructed with the LOD method. A superlocalization method known as Super-Localized Orthogonal Decomposition (SLOD) was introduced in \cite{HaPe21b} (see also {\cite{Bonizzoni-Freese-Peterseim,BHP22,pumslod,Freese-Hauck-Peterseim, PeWZi23, HaMoPe24})}, providing basis functions with superexponential decay whose supports are reduced in relation to their LOD counterpart. In this paper we first present a method to construct (almost-fully) $a$-orthogonal hierarchical basis functions localized by a novel variant of the SLOD method (based on stable corrections of LOD basis functions), and later we build a sparse-compressed operator that approximates the operator $\mathcal{A}^{-1}_h$ and therefore approximates $\mathcal{A}^{-1}$.

By combining the concepts of hierarchical $a$-orthogonal-across-levels basis and the SLOD method, the new method allows for a greater sparse compression of the solution space in comparison with the standard one-level SLOD, gamblets-based, and hierarchical LOD methods. This implies a greater sparsity of the associated stiffness matrix required to achieve a given accuracy (and a greater sparsity of the matrix obtained from the sparse-compressed approximations of its inverse, see Section 4), resulting in substantial savings in memory usage and online computational effort. Furthermore,  the new approach introduced to construct well-behaved superlocalized basis functions facilitates the computation of the sparse-compressed approximate-solution operator, where its approximation quality and computational cost benefit from the quality of the basis.

%Furthermore,  the new approach introduced to construct well-behaved superlocalized basis functions leads to a well-conditioned stiffness matrix (up to the contrast), which, in turn, benefits the approximation quality and computational cost of the sparse-compressed approximate-solution operator.

%By combining the concepts of hierarchical $a$-orthogonal (accross levels) basis and the SLOD method, the method proposed in this work allows for a greater sparsity in the compressed approximation of the stiffness matrix associated with the solution of (), with respect to that obtained with the standard SLOD, gamblets-based, or hierarchical LOD methods. Furthermore, by introducing a new approach for constructing well-behaved superlocalized basis functions, the method possibilitates the actual approximation of the solution operator itself by a computable sparse-compressed operator whose computational effort benefits from the well-behavior of the basis.

The remaining part of the paper is structured as follows. In \cref{Sec-Construction}, we present the construction of a localized $a$-orthogonal hierarchical basis. In \cref{Sec-Stable-corrections}, an LOD basis correction strategy to obtain a stable superlocalized hierarchical basis is provided. The construction of a sparse-compressed approximate-solution operator and the corresponding error analysis are given in \cref{Sec-Compressed-Op-construction-and-Error-Analysis}. Finally, \cref{sec:Numerical-Experiments} illustrates our theoretical findings through numerical experiments.

%We present next the construction of our hierarchical basis.

\section{Construction of hierarchical basis}\label{Sec-Construction}
In this section we focus on the construction of an $a$-orthogonal hierarchical basis. To that end, we provide first a procedure to obtain hierarchical bases of $\mathcal{V}_H$ whose basis functions regardless of their support are $a$-orthogonal across levels. Then we focus on localization strategies to obtain a basis of practical usage.

\subsection{{\emph a}-orthogonal bases of $\mathcal{V}_H$}\label{ss-Orth-bases}\label{subsec:A-orth-bases}
Let $\mathcal{T}_0$ denote the Cartesian mesh of $\Omega$ with mesh size $H_0$. Consider $H_0>H_1>H_2>\ldots>H_L>0$ (with $H_L=H$), and let $\{\mathcal{T}_{\ell}\}_{ \ell\in\{0,\ldots,L\}}$ be a set of Cartesian meshes of $\Omega$ obtained by successive refinements of $\mathcal{T}_{0}$. Furthermore, assume $H_{j+1}=H_j/2$ for $j\leq L-1$. 

Recall that $\mathcal{T}_h$ is a Cartesian mesh of $\Omega$ with mesh size $h$ and the operator $\mathcal{A}^{-1}_{h}$ is given by \eqref{Fine-problem-def}. Define $\mathcal{V}_{\ell}:=\mathbf{span}\{\mathcal{A}^{-1}_h \chi_{T}: T\in \mathcal{T}_{\ell}\}$, where $\chi_T$ is the characteristic function of $T$. In what follows, we derive a hierarchical basis of $\mathcal{V}_L\left(=\mathcal{V}_{H}\right)$.  Let $N_{\ell}:=\#\T_{\ell}$, and let $\mathcal{B}_{\ell}$ be the set of basis functions of $\mathcal{V}_{{L}}$ associated with level $\ell\leq L$. For $\ell=0$, we set $\#\mathcal{B}_0=N_{0}$. For $0< \ell \leq L$, the set $\mathcal{B}_{\ell}$ consists of $(2^d-1)N_{\ell-1}$ linearly independent functions that are $a$-orthogonal to every function in $\cup_{s=0}^{\ell-1}\mathcal{B}_{s}$. With this selection, $\cup_{s=0}^{\ell}\mathcal{B}_{s}$ is a basis of $\mathcal{V}_{\ell}$ for every $\ell$.

For any level $0\leq \ell \leq L$ and $i\in \{1,\ldots,\# \mathcal{B}_{\ell}\}$, we define the basis function  $\varphi_{\ell,i}\in \mathcal{B}_{\ell} \subset \mathcal{V}_{L}$ as 
\begin{equation}\label{varphi-levl2}
\varphi_{\ell,i}=\mathcal{A}^{-1}_{h}g_{\ell,i} \quad \mathrm{with}\quad g_{\ell,i} \coloneqq \sum_{K\in \mathcal{T}_{{\ell}}}c^{(\ell,i)}_{K}\chi_{K}.
\end{equation}
The non-trivial coefficients $c^{(\ell,i)}_{K}$ are chosen such that $\{g_{0,i}\}_{i\in\{1,\ldots,N_0\}}$ is a basis of $\mathbb{Q}^0(\mathcal{T}_{0})$, and for $\ell>0$, the $a$-orthogonality condition across levels $a(\varphi_{p,j},\varphi_{\ell,i})=0$ holds for $0\leq p<\ell$. Henceforth, we refer to $g_{\ell,i}$ as the $\mathbb{Q}^0$-companion of $\varphi_{\ell,i}$.

Investigating the $a$-orthogonality condition, we observe that
\begin{equation}\label{Lev1-ort-cond-start1}
	\nonumber a(\varphi_{p,j},\varphi_{\ell,i})= \left( g_{p,j},\varphi_{\ell,i} \right)_{L^{2}(\Omega)}=\sum_{T\in \T_p}c_T^{(p,j)}\int_{T}\bfideal_{\ell,i}\dx=0.
\end{equation}
Therefore, the $a$-orthogonality condition $a(\bfideal_{p,j},\bfideal_{\ell,i})=0$ is satisfied if
\begin{equation}\label{intT-phi-equ-0}
	\int_{T}\varphi_{\ell,i}\dx=0 \quad \text{for all } \,T\in \mathcal{T}_{\ell-1},
\end{equation}
since the integral over a mesh element belonging to $\mathcal{T}_{p}$ with $p\leq \ell-1$ is just a sum of integrals over mesh elements belonging to $\T_{\ell-1}$. Defining $\alpha_{T,K}:=\int_{T}\mathcal{A}_{h}^{-1}\chi_{K}\dx$ and using the definition of $\varphi_{\ell,i}$ from \eqref{varphi-levl2}, we can rewrite condition \eqref{intT-phi-equ-0} as 
\begin{equation}\label{Alpha-def-1}
	\sum_{K\in\T_\ell} c_K^{(\ell,i)} \alpha_{T,K} = 0  \quad \text{for all } \,T\in \mathcal{T}_{\ell-1}.
\end{equation}
Let $\mathbf{D^{(\ell)}} \in \mathbb{R}^{N_{\ell-1}\times N_{\ell}}$ such that $\mathbf{D}_{n,m}^{(\ell)}=\alpha_{T_n,K_m}$  with $T_n\in \T_{\ell-1}$, $K_m\in \T_{\ell}$ and $\mathbf{c}^{(\ell,i)}=\left(c_{K}^{(\ell,i)}\right)_{K\in \mathcal{T}_{\ell}}$. Hence, the $a$-orthogonality condition is satisfied for any $\ell>0$ if  
\begin{equation}\label{Global-a-orth-sufficient-cond} 
\mathbf{D}^{(\ell)}\mathbf{c}^{(\ell,i)}=\mathbf{0} \quad \text{for all } i\,\in \{1,\ldots,(2^d-1)N_{\ell-1}\},
\end{equation}
which implies that $\mathbf{c}^{(\ell,i)}\in \text{Ker}(\mathbf{D}^{(\ell)})$. Note that $\text{rank}(\mathbf{D}^{(\ell)})\leq N_{\ell-1}$. Consequently, $\text{dim}(\text{Ker}(\mathbf{D}^{(\ell)}))\geq N_{\ell}-N_{\ell-1}=(2^d-1)N_{\ell-1}$. In principle, we could take any set of $(2^d-1)N_{\ell-1}$ linearly independent vectors of $\text{Ker}(\mathbf{D}^{(\ell)})$ for level $\ell>0$ and any basis of $\mathbb{Q}^{0}(\mathcal{T}_0)$ to determine a basis of $\mathcal{V}_{{L}}$. Note, however, that the associated basis functions $\varphi_{\ell,i}$ will in general be globally supported and the corresponding stiffness matrix will not be sparse. Therefore, an ideal choice would be a set of coefficient vectors for which the associated functions $\varphi_{\ell,i}$ are locally supported. We study the feasibility of such an option in the following subsection. 

\begin{remark}
Enforcing condition  \eqref{Global-a-orth-sufficient-cond} is not the only path to achieve $a$-orthogonality across levels. Nevertheless, this is an advantageous option from the computational point of view.
\end{remark}

\subsection{Toward basis functions with local support}\label{ss-basis-local}
Ideally, we would like to find coefficients such that each previously defined basis function $\varphi_{\ell,i}$ is locally supported. In this subsection we investigate whether it is possible to obtain locally supported 
functions of $\mathcal{V}_{\ell}$ from locally supported $\mathbb{Q}^0$-companions. To explore this possibility, we need a series of definitions, some of which are adaptations of ideas found in \cite{BHP22}.

For $\ell\in\{0,\ldots,L\}$, $m\in \mathbb{N}_{0}$, $j\in\{1,\ldots,N_{\ell-1}\}$, the $m$-th order patch $\omega_j^{(\ell,m)}\subset\Omega$  of an element $T_j\in \mathcal{T}_{{\ell-\min\{\ell,1\}}}$ is defined recursively as
\begin{equation*}
\omega_{j}^{(\ell,m)}=\bigcup \{T\in \mathcal{T}_{{\ell-\min\{\ell,1\}}}: T\cap \omega^{(\ell,m-1)}_{j}\neq \emptyset\} \quad \mathrm{with}\quad \omega^{(\ell,0)}_{j}=T_j\in \mathcal{T}_{{\ell-\min\{\ell,1\}}}.
\end{equation*}
Thus, for level $\ell=0$, the patches are unions of elements in $\mathcal{T}_{\ell}$, and for $\ell\geq1$ they are unions of elements in $\mathcal{T}_{\ell-1}$.

In what follows, we fix $m\geq1$, $0 \leq \ell \leq L$, and $T_j\in\mathcal{T}_{\ell-\min\{\ell,1\}}$, and let $\omega:=\omega_{j}^{(\ell,m)}$. The restriction of $\mathcal{T}_{\ell}$ to the patch $\omega$ is denoted by $\mathcal{T}_{\ell,\omega}:=\{T\in\mathcal{T_{\ell}}:T\subset \omega\}$. Similarly, the restriction of $V_h$ to $\omega$ is given by $V_{h,\omega}:=\{v\in H^1_0(\omega):v_{|_{T}}\in \mathbb{Q}^{1}(T)\; \text{for all } T\in \mathcal{T}_{h,\omega}\}$. Let $\widetilde{V}_{h,\omega}:=\{v\in H^1(\omega):v_{|_{T}}\in \mathbb{Q}^{1}(T)\; \text{for all } T\in \mathcal{T}_{h,\omega}\}$. Thus, the trace of a function in $\widetilde{V}_{h,\omega}$ may be non-zero. With $\Sigma=\partial \omega \setminus \partial \Omega$, we define the trace operator
\begin{equation*}
\text{tr}_{\Sigma}: \widetilde{V}_{h,\omega}\rightarrow X_{\omega}:=\text{image tr}_{\Sigma}\subset H^{1/2}(\Sigma).
\end{equation*}
Note that the space $X_{\omega}$ can be equipped with the norm
\begin{equation*}
\|w\|_{X_{\omega}}:=\text{inf}\left\{ \|v\|_{H^{1}(\omega)}: v\in \widetilde{V}_{h,\omega}, \; \text{tr}_{\Sigma}v=w \right\}.
\end{equation*}
By definition of the $\|\cdot\|_{X_{\omega}}$ norm, the continuity of the trace operator holds regardless of the patch geometry, i.e., 
\begin{equation}\label{trace-norm-continuity}
\|\text{tr}_{\Sigma}v\|_{X_{\omega}}\leq\|v\|_{H^{1}(\omega)} \;\;\;\;\;\;\text{for all } v\in \widetilde{V}_{h,\omega}\subset H^1(\omega).
\end{equation}

Let $a_{\omega}(\cdot,\cdot):H^{1}(\omega) \times H^{1}(\omega) \rightarrow \mathbb{R}$ such that $a_\omega(u,v)\coloneqq\int_{\omega}(\mathbf{A}\nabla u) \cdot \nabla v \dx$,
and let the linear operator $\mathcal{A}^{-1}_{h,\omega}:L^{2}(\omega)\rightarrow V_{h,\omega}$ be given by
\begin{equation*}
a_\omega(\mathcal{A}^{-1}_{h,\omega} g,v )=(g,v)_{L^2(\omega)},\;\;\; \text{for all } v\in V_{h,\omega}, \text{ and } g\in L^{2}(\omega).
\end{equation*}
Denote by $\mathcal{E}_\omega: L^2(\omega) \rightarrow L^2(\Omega)$ %\chcomment{Doesn't make sense. Do you mean $\mathcal{E}_\omega: L^1(\omega) \to L^1(\Omega)?$} 
an extension by zero operator. Then, with $g=\sum_{T\in \mathcal{T}_{{\ell},\omega}}c_{T}\chi_{T}\in \mathbb{Q}^{0}(\mathcal{T}_{{\ell},\omega})$, consider 
\begin{equation*}
 \bar{\psi}_{g}:=\mathcal{A}^{-1}_{h,\omega}g \in \mathcal{V}_{\ell,\omega}\quad \mathrm{and}\quad \psi_g:=\mathcal{A}^{-1}_{h}\mathcal{E}_\omega(g) \in \mathcal{V}_{\ell},
\end{equation*}
where $\mathcal{V}_{\ell,\omega}=\{\mathcal{A}^{-1}_{h,\omega}\chi_T : T\in \mathcal{T}_{\ell,\omega}\}$.
We conclude the set of definitions by defining the conormal derivative of $\bar{\psi}_g$ as the functional $\gamma_{\bar{\psi}_g}=\mathbf{A}\nabla \bar{\psi}_g \cdot \mathbf{n}:X_\omega \rightarrow \mathbb{R}$ such that
\begin{equation}\label{Def-Conormal-Der}
\langle \gamma_{\bar{\psi}_g},\text{tr}_{\Sigma}(v) \rangle=a_{\omega}(\bar{\psi}_g,v)-(g,v)_{L^2(\omega)}\;\;\;\;\text{for all } v\in H^{1}(\omega).
\end{equation}
Note that $\mathcal{E}_{\omega}(\bar{\psi}_g)$ does not necessarily belong to $\mathcal{V}_{\ell}$. However, if $g$ is such that $\mathcal{E}_{\omega}(\bar{\psi}_g)=\psi_g$, then $\psi_g$ would be a locally supported function of $\mathcal{V}_{\ell}$ whose $\mathbb{Q}^{0}$-companion $\mathcal{E}_{\omega}(g)$ is also locally supported.

In the following lemma, we determine a bound for the energy norm of the {\it localization error} $\mathcal{E}_{\omega}(\bar{\psi}_{g})-\psi_g$, which depends on the $X_{\omega}'$-norm of $\bar{\psi}_{g}$ given by
\begin{equation}\label{ConormalDer-Norm-definition}
\|\gamma_{\bar{\psi}_{g}}\|_{X_{\omega}'}=\sup_{w\in X_\omega\setminus \{0\}}\frac{\langle \gamma_{\bar{\psi}_{g}},w \rangle}{\|w\|_{X_\omega}},
\end{equation}
providing a way to measure the dependence of the error on the coefficients defining $g$.
\begin{lemma}\label{Lemma-Localization-Error-bound-conormalDer}
Let $\bar{\psi}_{g}$, $\psi_g$, and $\gamma_{\bar{\psi}_{g}}$ be defined as above. Then, the energy norm of the localization error has the bound
\begin{equation}\label{Loc-Error-bound-conormalDer}
\|\mathcal{E}_\omega(\bar{\psi}_{g})-\psi_g\|_{a}\leq \frac{1+\frac{\mathrm{diam}(\Omega)}{\pi}}{\sqrt{\alpha}}\|\gamma_{\bar{\psi}_{g}}\|_{X_{\omega}'}
\end{equation}
where $\mathrm{diam}(\Omega)$ denotes the diameter of $\Omega$ and the constant $\alpha$ is given in (\ref{A-spectral-bound}).
\end{lemma}
\begin{proof}
We have $\text{for all } v \in H^{1}(\Omega)$ that
\begin{eqnarray}\label{Loc-Error-bound-conormalDer-prelim-result}
\nonumber a(\mathcal{E}_{\omega}(\bar{\psi}_{g})-\psi_g,v)&=&a_{\omega}(\bar{\psi}_g,v_{|_{\omega}})-(g,v_{|_{\omega}})_{L^2(\omega)} \\ \nonumber
&=&\langle \gamma_{\bar{\psi}_{g}},\text{tr}_{\Sigma}(v_{|_{\omega}}) \rangle \\ \nonumber
&\leq & \|\gamma_{\bar{\psi}_{g}}\|_{X_{\omega}'} \|\text{tr}_{\Sigma}(v_{|_{\omega}})\|_{X_{\omega}} \\ \nonumber
&\leq &\|\gamma_{\bar{\psi}_{g}}\|_{X_{\omega}'}\|v_{|_{\omega}}\|_{H^{1}(\omega)}\\ 
&\leq & \frac{1+\frac{\mathrm{diam}(\omega)}{\pi}}{\sqrt{\alpha}}\|\gamma_{\bar{\psi}_{g}}\|_{X_{\omega}'} \|v_{|_{\omega}}\|_{a_{\omega}},
\end{eqnarray}
where the first inequality comes from the definition of the $X'_{\omega}$-norm, the second inequality from (\ref{trace-norm-continuity}), and the last inequality is obtained using Poincar\'e-Friedrichs inequality and the spectral bound of $\mathbf{A}$ from (\ref{A-spectral-bound}). 

Then, with $v=\mathcal{E}_{\omega}(\bar{\psi}_{g})-\psi_g$ in \eqref{Loc-Error-bound-conormalDer-prelim-result}, and since $\|v_{|_{\omega}}\|_{a_{\omega}}\leq \|v\|_{a}$ and $\mathrm{diam}(\omega)\leq \mathrm{diam}(\Omega)$, the inequality \eqref{Loc-Error-bound-conormalDer} is obtained.
\end{proof}

It follows from Lemma \ref{Lemma-Localization-Error-bound-conormalDer} that, if there exists a non-trivial $g\in \mathbb{Q}^0(\mathcal{T}_{\ell,\omega})$ such that $\|\gamma_{\bar{\psi}_{g}}\|_{X_{\omega}'}$ equals $0$, then there exists a locally supported function of $\mathcal{V}_{\ell}$ with locally supported $\mathbb{Q}^0$-companion $\mathcal{E}_\omega(g)$.

\subsection{A practical hierarchical basis}\label{sec:practical_basis}

Combining the results from the previous two subsections, we want to find local $\mathbb{Q}^0$ functions such that their local regularized companions (in $\mathcal{V}_{\ell,\omega}$) not only satisfy the $a$-orthogonality condition across levels but also have zero conormal derivative. However, as it can be observed in practice, it is not always possible to find local $\mathbb{Q}^0$ functions such that the above two conditions are satisfied exactly. In those cases, whether it is still possible to obtain locally-supported basis functions of $\mathcal{V}_{\ell}$ from the linear combination of globally-supported functions remains an open question. However, even if this is possible, it might not be the best option from the computational perspective. Thus, to obtain a computationally more advantageous scheme which only requires local computations, we might have to find a compromise between satisfying the above two sufficient conditions exactly and computational effort, at the expense of small errors. In what follows, we derive a practical hierarchical basis which is a good compromise, and in certain cases even fulfills those design ideals.

As mentioned in \Cref{subsec:A-orth-bases}, we need to find at every level $N_\ell^b$ functions, with
\begin{equation*}
	N_\ell^b \coloneqq \begin{cases}
		N_0, & \text{if }\ell=0,\\
		(2^d-1)N_{\ell-1}, & \text{if }\ell>0,
	\end{cases}
\end{equation*} which are linearly independent and $a$-orthogonal to every basis function associated with a coarser level.

Since the level of the patches can be inferred from the level of functions defined in them, for the sake for readability, we henceforth drop the superindices $(\ell,m)$ from $\omega_{j}^{(\ell,m)}$ and refer to the patches $\omega_{j}^{(\ell,m)}$ as $\omega_{j}$. Let $J_i:=\left\lfloor \frac{i-1} {(2^d-1)^{\min\{\ell,1\}}} \right \rfloor +1$. With $\mathbf{c}^{(\ell,i)}=\left(c_K\right)_{K\in\mathcal{T}_{\ell,\omega_{J_i}}}$, we define the local functions
\begin{equation}\label{barpsi-local-def}
\bar{\psi}_{\ell,i}=\bar{\psi}_{\ell,i}(\cdot;\mathbf{c}):=\mathcal{A}_{h,\omega_{J_i}}^{-1}\sum_{K\in \mathcal{T}_{{\ell},\omega_{J_i}}}c_K \chi_{K}
\end{equation} 
for $\ell=\{0,\ldots,L\}$ and $i\in\{1,\ldots,N^{b}_{\ell}\}$.

\begin{remark}
Note that $J_i$ denotes the index of the patch in which the function $\bar{\psi}_{\ell,i}$ is defined. By definition, $J_i=J_j=k$ whenever $(k-1)(2^d-1)+1< i,j \leq k(2^d-1)$ for $k\in \{1,\ldots,N_{\ell-1}\}$. Consequently, there are $2^d-1$ functions $\bar{\psi}_{\ell,i}$ defined in the patch $\omega_{k}$. This definition of $J_i$ is motivated by the idea that, in the minimal patch case (i.e., $m=0$), there are at most $2^d-1$ linearly independent functions $\bar{\psi}_{\ell,i}$ defined in a patch $\omega_{k}$ such that the orthogonality condition $a(\mathcal{E}_{\omega_{k}}(\bar{\psi}_{\ell,i}),\mathcal{E}_{\omega_{J_j}}(\bar{\psi}_{p,j}))=0$ holds for $0\leq p <\ell$, $j\leq N^{b}_{p}$, and $\omega_{k}\subset \omega_{J_j}$ (cf. \eqref{linsys-coeffs-local} below).
\end{remark}

Analogously to \eqref{intT-phi-equ-0}, at level $\ell>0$ we choose the coefficients $c_K^{(\ell,i)}$ for $i=1,\ldots,(2^d-1)N_{\ell-1}$ such that
\begin{equation}\label{varphi-zero-integral}
\int_{T}\bar{\psi}_{\ell,i}\dx=0\;\;\;\; \text{for all } \,T\in \T_{{\ell-1},\omega_{J_i}},
\end{equation}
which is equivalent to the condition
 \begin{equation}\label{linsys-coeffs-local}
 \sum_{K\in T_{{\ell},\omega_{J_i}}} c_K \alpha_{T,K}^{(J_i)}=0\;\;\;\; \text{for all } \,T\in \T_{{\ell-1},\omega_{J_i}}
 \end{equation}
  with
 \begin{equation}\label{Alpha-local}
 \alpha_{T,K}^{(J_i)}:=\int_{T}\mathcal{A}^{-1}_{h,\omega_{J_i}}\chi_{K}\dx.
 \end{equation}
This condition ensures the $a$-orthogonality of $\mathcal{E}_{\omega_{J_i}}(\bar{\psi}_{\ell,i})$, $\mathcal{E}_{\omega_{J_j}}(\bar{\psi}_{p,j})$, for $p\neq \ell$, whenever the support of the function from the higher level falls within the support of the one from the lower level.
\begin{remark}\label{rmk:a-Inner-Product-value}
If the functions $\hat{\varphi}_{\ell,i}$, $\hat{\varphi}_{p,j}$, with $p< \ell$ are such that the support of $\hat{\varphi}_{\ell,i}$ partially overlaps that of $\hat{\varphi}_{p,j}$, it follows that $\left(\hat{\varphi}_{p,j}\right)_{|_{\omega_{J_i}}}\notin H^{1}_{0}(\omega_{J_i})$, and consequently
\begin{equation*}
a(\hat{\varphi}_{\ell,i},\hat{\varphi}_{p,j})=\int_{\partial \omega_{J_i}}\hat{\varphi}_{p,j}\left(\mathbf{A}\nabla \hat{\varphi}_{\ell,i}\cdot \mathbf{n} \right)ds.
\end{equation*}
Since the basis functions of $\hat{\mathcal{B}}$ can be chosen to be uniformly bounded in the energy norm by construction, functions with partially overlapping support get closer to be $a$-orthogonal (i.e., the quantity $a(\hat{\varphi}_{\ell,i},\hat{\varphi}_{p,j})$ gets closer to zero) as $\mathbf{A}\nabla \hat{\varphi}_{\ell,i}\cdot \mathbf{n}$ decreases.
\end{remark}

Let $n_r=\#\T_{\ell-1,\omega_{J_i}}$, $n_c=\#\T_{\ell,\omega_{J_i}}$, and let $\mathbf{D}^{(\ell,J_i)}\in \mathbb{R}^{n_r \times n_c}$ be such that $\mathbf{D}^{(\ell,J_i)}_{q,s}=\alpha^{(J_i)}_{T_q,K_s}$, with $T_q\in \mathcal{T}_{\ell-1}$, $K_s \in \mathcal{T}_{\ell}$. Then equation (\ref{linsys-coeffs-local}) can be rewritten as
\begin{equation}\label{linsys-mat-form}
\mathbf{D}^{(\ell,J_i)}\mathbf{c}^{(\ell,i)}=\mathbf{0},
\end{equation}
which implies $\mathbf{c}^{(\ell,i)}\in \text{Ker}\left(\mathbf{D}^{(\ell,J_i)}\right)$. 

To complete the definition of our practical basis, it suffices to find a non-trivial $\mathbf{c}^{(\ell,i)}$ such that $\| \gamma_{\bar{\psi}_{\ell,i}(\cdot;\mathbf{c}^{(\ell,i)})}\|_{X'_{\omega_{J_i}}}$ is as small as possible (ideally zero), and then take $\hat{\varphi}_{\ell,i}=\mathcal{E}_{\omega_{J_i}}\left(\bar{\psi}_{\ell,i}(\cdot;\mathbf{c}^{(\ell,i)})\right)$ as the basis function.

Therefore, summarizing our previous discussion, we can define a more practical discrete solution space of (\ref{Continuous-Formulation}), henceforth denoted by $\hat{\mathcal{V}}_{L}$, as the span of the hierarchical basis  $\hat{\mathcal{B}}=\cup_{\ell=0}^{L} \hat{\mathcal{B}}_{\ell}$, where $\hat{\varphi}_{\ell,i}\in \hat{\mathcal{B}}_{\ell}$ is given by
\begin{equation*}
\hat{\varphi}_{\ell,i}=\mathcal{E}_{\omega_{J_i}}\left(\bar{\psi}_{\ell,i}(\cdot;\mathbf{c}^{(\ell,i)})\right),
\end{equation*}
with $\mathbf{c}^{(\ell,i)}\neq\mathbf{0}$ such that $\| \gamma_{\bar{\psi}_{\ell,i}(\cdot,\mathbf{c}^{(\ell,i)})}\|_{X'_{\omega_{J_i}}}$ is as small as possible, and with the additional condition for $\ell>0$ that $\mathbf{c}^{(\ell,i)}\in \text{Ker}\left(\mathbf{D}^{(\ell,J_i)}\right)$.

In practice, we observed that if we take the local coefficients
\begin{equation}\label{Original-Optim-problem}
\mathbf{c}^{(\ell,i)}=\argmin_{\substack{\mathbf{c}\in \mathbb{R}^{n_c},\\ \|\mathbf{c}\|_{\mathbf{M}}=1}} \frac{1}{2} \|\gamma_{\bar{\psi}_{i,\ell}}\|_{X_{\omega_{J_i}}'}^2,
\end{equation}
where $\|\mathbf{c}\|_{\mathbf{M}}=\sqrt{\mathbf{c}^T\mathbf{M}\mathbf{c}}$ for some SPD matrix $\mathbf{M}$ (e.g., the stiffness or mass matrix associated with $\{\mathcal{A}^{-1}_{h,\omega_{J_i}}\chi_{K}\}_{K\in T_{{\ell},\omega_{J_i}}}$, or the identity matrix), we end in some cases with a Riesz stability-deficient basis, i.e., the associated stiffness matrix is ill-conditioned regardless of the values of $\beta$ and $\alpha$ (cf. \Cref{Remark-basis-well-behavior-measure} below). To overcome this poor stability issue, we shall introduce in the next section a method inspired by the LOD method to obtain coefficients that allow the construction of a Riesz stable basis that preserves the smallness of the localization errors and their superexponential decay.

%\begin{remark}
%Note from (\ref{Alpha-def-1}), (\ref{Alpha-local}), and the definitions of $\mathbf{D}^{(\ell)}$ and $\mathbf{D}^{(\ell,J_i)}$, that the prolongation by zeros of $\mathbf{c}^{(\ell,i)}\in \mathbb{R}^{n_c}$ to $\mathbb{R}^{N_{\ell}}$ is not necessarily an element of $\text{Ker}(\mathbf{D}^{(\ell)})$. Consequently, with $\mathcal{B}_{\ell}=\{\psi_{\ell,i}\}_{i=1}^{N_{\ell}^{b}}$, the stiffness matrix associated with the basis $\mathcal{B}=\cup_{\ell=0}^{L} \mathcal{B}_{\ell}$ (a basis of $\mathcal{V}_{L}$) will in general not be a block-diagonal matrix.
%\end{remark}

\section{Stable localization of basis functions}\label{Sec-Stable-corrections}
In this section, we present a localization strategy that leads to a practically stable basis. The method involves constructing a superlocalized basis (hereafter SLOD basis) at each level by correcting LOD basis functions, and then defining the hierarchical basis functions at that level (hereafter HSLOD basis functions) as a linear combination of selected SLOD basis functions. 
\subsection{A stable SLOD basis for level $\ell$}\label{Stable-SLOD-subsection}
In what follows, we derive a stable SLOD basis for a fixed level $\ell$ from the correction of local LOD basis functions. The patches defining the local LOD basis functions differ from those previously introduced in the definition of hierarchical basis functions in that, for all $\ell$, the patches defining LOD basis functions at level $\ell$ are unions of elements of $\mathcal{T}_{\ell}$. To make this difference clear, we shall denote the LOD patches with a tilde. Hence, we define recursively the $m$-th order LOD patch associated with $T_j\in \mathcal{T}_{\ell}$ as
\begin{equation*}
\widetilde{\omega}_{j}^{(\ell,m)}=\bigcup \{T\in \mathcal{T}_{{\ell}}: T\cap \widetilde{\omega}^{(\ell,m-1)}_{j}\neq \emptyset\} \quad \mathrm{with}\quad \widetilde{\omega}^{(\ell,0)}_{j}=T_j.
\end{equation*}
{Note for $\ell>0$ that, as a difference from $\widetilde{\omega}_{j}^{(\ell,m)}$, the patch $\omega_{j}^{(\ell,m)}$ is constructed around an element of $\mathcal{T}_{\ell-1}$ with layers composed of elements of $\mathcal{T}_{\ell-1}$ instead of elements of $\mathcal{T}_{\ell}$, see \cref{fig:patches}.}
\begin{figure}
	\centering
	\includegraphics[width=.35\linewidth,height=0.325\linewidth]{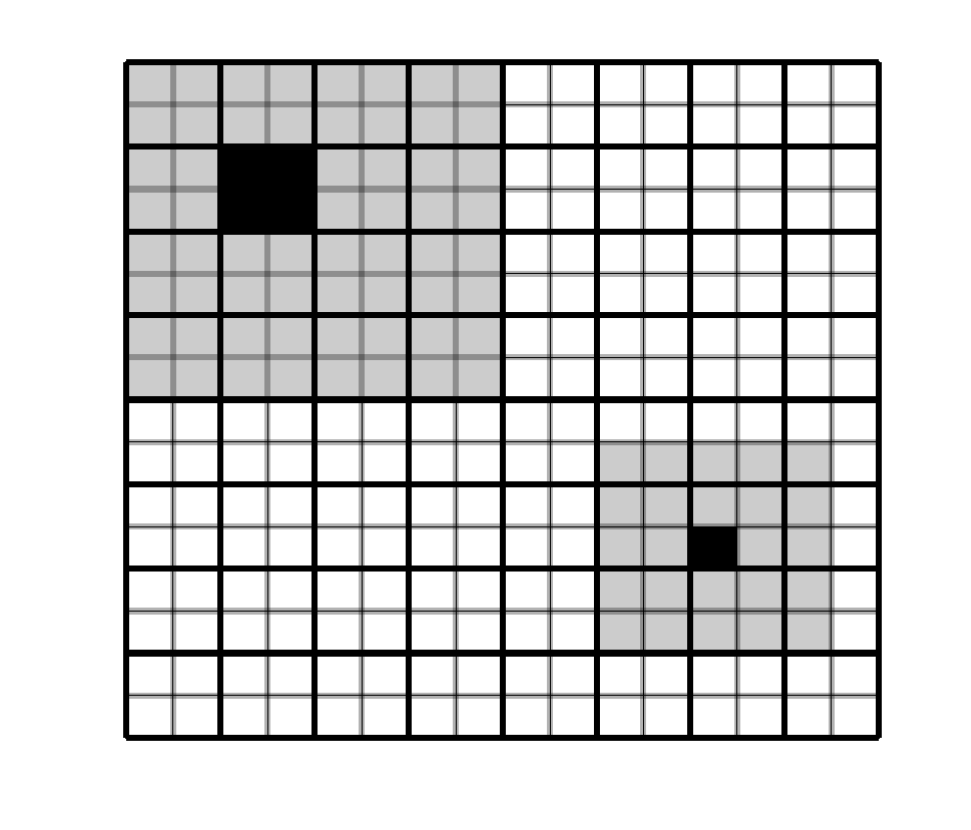}
	\caption{Illustration of the two second-order patches $\omega^{(\ell,2)}_T$ (top left) and  $\widetilde{\omega}^{(\ell,2)}_K$ on the mesh $\T_\ell$ for some $\ell>0$. The patches are centered around some mesh elements $T\in\T_{\ell-1}$ and $K\in\T_\ell$, respectively. Additionally, the coarser mesh $\T_{\ell-1}$ is depicted with bold lines.}
	\label{fig:patches}
\end{figure}

With fixed values of $\ell$, $m$, and $T=T_j\in \mathcal{T_{\ell}}$, let $\widetilde{\omega}:=\widetilde{\omega}_{j}^{(\ell,m)}$. Define $\Pi_{\ell,\tilde{\omega}}$ as the restriction of $\Pi_{\ell}$ to $\tilde{\omega}$.  Using the energy-minimization saddle-point formulation of the LOD method \cite{Mai20ppt,OwhS19}, we have for every $T\in \mathcal{T}_{\ell,\widetilde{\omega}}$ that 
\begin{equation}\label{LOD-Saddle-point-formulation}
\begin{pmatrix}
\mathcal{A}_{h,\widetilde{\omega}} & \mathcal{P}^T\\
\mathcal{P} & 0\\
\end{pmatrix}
\begin{pmatrix}
\bar{\psi}^{\text{LOD}}_{T,\ell}\\
\lambda
\end{pmatrix}= \begin{pmatrix}
0\\
\chi_{T}
\end{pmatrix},
\end{equation}
where $\mathcal{A}_{h,\widetilde{\omega}}:V_{h,\widetilde{\omega}}\rightarrow \left[V_{h,\widetilde{\omega}}\right]'$, $u \mapsto a_{\widetilde{\omega}}(u,\cdot)$, $\mathcal{P}:V_{h,\widetilde{\omega}}\rightarrow \mathbb{Q}^{0}(\mathcal{T}_{\ell,\widetilde{\omega}})$, $u \mapsto \Pi_{\ell,\widetilde{\omega}}u$, and $\mathcal{P}^{T}:\mathbb{Q}^{0}(\mathcal{T}_{\ell,\widetilde{\omega}})\rightarrow \left[V_{h,\widetilde{\omega}}\right]'$ such that $\langle \mathcal{P}^{T} p, v \rangle = (p,v)_{L^2(\widetilde{\omega})}$ for all $p\in \mathbb{Q}^{0}(\mathcal{T}_{\ell,\widetilde{\omega}})$ and $v\in V_{h,\widetilde{\omega}}$. Let $\mathcal{D}:\mathbb{Q}^{0}(\mathcal{T}_{\ell,\widetilde{\omega}})\rightarrow \mathbb{Q}^{0}(\mathcal{T}_{\ell,\widetilde{\omega}})$ such that $\mathcal{D}=\mathcal{P}\circ \mathcal{A}_{h,\widetilde{\omega}}^{-1}\circ \mathcal{P}^T$. It can be shown that $\mathcal{D}$ is invertible; see \cite{HaPe21b}. Thus, eliminating $\lambda$ in (\ref{LOD-Saddle-point-formulation}) and solving for the local LOD basis function $\psi^{\text{LOD}}_{T_i,\ell}$ yields
\begin{equation*}
\bar{\psi}^{\text{LOD}}_{T,\ell}=\mathcal{A}^{-1}_{h,\widetilde{\omega}}g_{T,\ell},
\end{equation*}
where $g_{T,\ell} =\mathcal{D}^{-1}\chi_{T}$.

For $K\in \mathcal{T}_{\ell,\widetilde{\omega}}\setminus \{T\}$, let $\bar{\psi}_{K,\ell}^{(T)}:=\mathcal{A}^{-1}_{h,\tilde{\omega}}g_{K,\ell}$, where $g_{K,\ell}=\mathcal{D}^{-1}\chi_{K}$. With $\mathbf{c}=(c_K)_{K\in \mathcal{T}_{{\ell},\widetilde{\omega}}\setminus\{T\}}$, define $\Psi_{\mathbf{c}}^{(T)}\in V_{h,\widetilde{\omega}}$ as
\begin{equation*}
\Psi_{\mathbf{c}}^{(T)}:=\sum_{K\in \mathcal{T}_{{\ell},\widetilde{\omega}}\setminus\{T\}}c_K \bar{\psi}_{K,\ell}^{(T)}=\mathcal{A}^{-1}_{h,\widetilde{\omega}}\sum_{K\in \mathcal{T}_{{\ell},\widetilde{\omega}}\setminus\{T\}}c_K g_{K,\ell}.
\end{equation*}
Let $n_e=\#T_{{\ell},\widetilde{\omega}}$. We want to find $\bar{\mathbf{c}}\in \mathbb{R}^{n_e-1}$ such that $\|\gamma_{\bar{\psi}^{\text{LOD}}_{T,\ell}+\Psi_{\bar{\mathbf{c}}}^{(T)}}\|_{X'_{\widetilde{\omega}}}$ is as small as possible (ideally zero), and then define the (normalized) SLOD basis function at level $\ell$ associated with the patch around $T$ as
\begin{equation*}
\hat{\theta}^{\text{SLOD}}_{T,\ell}=\frac{\mathcal{E}_{\widetilde{\omega}}\left(\bar{\psi}^{\text{LOD}}_{T,\ell}+\Psi_{\bar{\mathbf{c}}}^{(T)}\right)}{\|\bar{\psi}^{\text{LOD}}_{T,\ell}+\Psi_{\bar{\mathbf{c}}}^{(T)}\|_{a_{\widetilde{\omega}}}}\cdot
\end{equation*}
To make $\|\gamma_{\bar{\psi}^{\text{LOD}}_{T,\ell}+\Psi_{\mathbf{c}}^{(T)}}\|_{X'_{\widetilde{\omega}}}$ small, we could use a similar technique as in \cite{BHP22} to compute the $\|\cdot\|_{X'_{\widetilde{\omega}}}$ norm, and then minimize it over the coefficients $\mathbf{c}\in \mathbb{R}^{n_e-1}$ to obtain $\bar{\mathbf{c}}$. However, for computational-cost efficiency, and based on \eqref{ConormalDer-Norm-definition}, we rather seek $\bar{\mathbf{c}}\in \mathbb{R}^{n_e-1}$ such that 
\begin{equation}\label{Weekened-boundary-condition}
\sum_{i\in I_{\Sigma_{\widetilde{\omega}}}}\left\langle \gamma_{\bar{\psi}^{\text{LOD}}_{T,\ell}+\Psi_{\bar{\mathbf{c}}}^{(T)}},\mathrm{tr}_{\Sigma_{\widetilde{\omega}}}\phi_i \right\rangle ^2\leq\epsilon
\end{equation}
for a small $\epsilon\geq 0$, where $\Sigma_{\widetilde{\omega}}:=\partial \widetilde{\omega}\setminus \partial \Omega$, $I_{\Sigma_{\widetilde{\omega}}}:=\{i\in \mathbb{N}:x_i\in \Sigma_{\widetilde{\omega}}\}$, $x_i$ is the $i$-th nodal point associated with $\mathcal{T}_{h,\widetilde{\omega}}$, and $\phi_i$ is the $i$-th $\mathcal{Q}_{1}$ standard basis function of $V_h$.

\begin{remark}
Optimizing $\|\gamma_{\bar{\psi}^{\text{LOD}}_{T,\ell}+\Psi_{\mathbf{c}}^{(T)}}\|_{X'_{\widetilde{\omega}}}$ over $\mathbf{c}\in \mathbb{R}^{n_e-1}$ is similar to the optimization problem \eqref{Original-Optim-problem} in the sense that both problems seek to find coefficients that lead to functions with small conormal derivatives. However, they are not the same optimization problem since they have different constraints. Furthermore, the admissible sets of both optimization problems are vector spaces with different dimensions.
\end{remark}

\begin{remark}
In what follows, the parenthetical $T$ in a superscript refers to a mesh element $T$. Whenever $T$ appears without parenthesis in a superscript, it denotes the transpose sign.
\end{remark}

With $n_b=\#I_{\Sigma_{\widetilde{\omega}}}$, define the matrix $\mathbf{B}\in \mathbb{R}^{n_{b} \times n_e}$ such that
\begin{equation}\label{Bmatrix-TraceCondition}
\mathbf{B}_{ij} = a_{\widetilde{\omega}}(\mathcal{A}^{-1}_{\widetilde{\omega}}\chi_{T_j},\phi_i)-(\chi_{T_j},\phi_i)_{L^2(\widetilde{\omega})}.
\end{equation}
Let $g_{\tau,\ell}=\sum_{K\in \mathcal{T}_{{\ell,\widetilde{\omega}}}}d_{K}^{(\tau)}\chi_K$, with $\tau \in \mathcal{T}_{\ell,\widetilde{\omega}}$, $\mathbf{d}^{(\tau)}=\big(d_K^{(\tau)}\big)_{K\in \mathcal{T}_{{\ell},\widetilde{\omega}}}$, and $\mathbf{D}\in \mathbb{R}^{n_e \times (n_e-1)}$ such that the $j$-th column of $\mathbf{D}$ is $\mathbf{d}^{(\tau_j)}$, with $\tau_j\in \mathcal{T}_{{\ell},\widetilde{\omega}}\setminus\{T\}$. Then, the $\bar{\mathbf{c}}$ providing the smallest $\epsilon$ in (\ref{Weekened-boundary-condition}) is the least-squares-error solution of
\begin{equation}\label{linear-system-solved-by-least-squares}
\mathbf{B}\mathbf{D}\mathbf{c}=-\mathbf{B}\mathbf{d}^{\left(T\right)}
\end{equation}
i.e.,
\begin{equation}\label{c-lsq-solution}
\bar{\mathbf{c}}=-\left((\mathbf{B}\mathbf{D})^{T}\mathbf{B}\mathbf{D}\right)^{-1}(\mathbf{B}\mathbf{D})^{T}\mathbf{B}\mathbf{d}^{\left(T\right)}.
\end{equation}
Note that, to guarantee stability of the SLOD basis, we additionally want $\Psi_{\bar{\mathbf{c}}}$ to be such that
\begin{equation}\label{SLOD-stability-condition}
\Big\|\frac{\Pi_{{\ell}}\hat{\theta}^{\text{SLOD}}_{T,\ell}}{z_{T}}-\chi_{T}\Big\|_{L^{\infty}(\Omega)}\leq\delta_s,
\end{equation} 
with
\begin{equation}\label{SLOD-stability-condition-weight}
z_{T}=\frac{\big(\Pi_{{\ell}}\hat{\theta}^{\text{SLOD}}_{T,\ell},\chi_{T}\big)_{L^2(\widetilde{\omega})}}{\left(\chi_{T},\chi_{T}\right)_{L^2(\widetilde{\omega})}}
\end{equation}
and $\delta_s\geq 0$ small (in practice, we observed that taking $\delta_s\leq0.5$ is enough to obtain basis stability). Note that condition \eqref{SLOD-stability-condition} impels the support of $\hat{\theta}^{\text{SLOD}}_{T,\ell}$ to be reasonably concentrated around $T$, a favorable property for basis stability; see Appendix \ref{Appendix-SLOD-Stability} for a more detailed explanation of why condition \eqref{SLOD-stability-condition} ensures basis stability. Now, in practice it is observed that choosing $\bar{\mathbf{c}}$ as in (\ref{c-lsq-solution}) does not always satisfy condition (\ref{SLOD-stability-condition}). Hence, for the sake of stability, we choose $\bar{\mathbf{c}}$ instead as follows. Expressing $(\mathbf{B}\mathbf{D})^{T}\mathbf{B}\mathbf{D}$ in terms of its singular value decomposition we have
\begin{equation}\label{svd-computation}
(\mathbf{B}\mathbf{D})^{T}\mathbf{B}\mathbf{D}=\sum_{i=1}^{r}\sigma_i u_i v_i^T.
\end{equation}
where $\sigma_i$ is the $i$-th singular value, with $\sigma_1\geq\ldots \geq\sigma_r$, $u_i$ is the $i$-th left singular vector, $v_i$ is the $i$-th right singular vector, and $r$ is the rank of $(\mathbf{B}\mathbf{D})^{T}\mathbf{B}\mathbf{D}$. Then, we can make a stable choice of $\bar{\mathbf{c}}$ by taking 
\begin{equation}\label{stable_c_computation}
\bar{\mathbf{c}}_{s}=-\Big(\sum_{i=1}^{r_s^{(T)}}\sigma_i^{-1} v_i u_i^T\Big)(\mathbf{B}\mathbf{D})^{T}\mathbf{B}\mathbf{d}^{\left(T\right)},
\end{equation}
where $r_s^{(T)}$ is chosen so that condition (\ref{SLOD-stability-condition}) holds. Note that taking $r_s^{(T)}=0$ for all $T\in \mathcal{T}_{{\ell}}$ yields the LOD basis, which is known to be Riesz stable (see, e.g., \cite{HaPe21b}). The LOD basis functions equivalently satisfy the stability condition (\ref{SLOD-stability-condition}) exactly with $\delta_s=0$, by definition. Therefore, there always exists at least one value of $r_s^{(T)}$ for which a stable basis can be obtained using this stabilization procedure. However, if $r_s^{(T)}$ is too small, the superexponential decay of the basis functions might be lost, as expected. Nevertheless, even in such cases, the basis functions still exhibit exponentially decaying properties. Hence, we want to choose $\delta_s$ just small enough to achieve basis stability and such that $r_s^{(T)}$ remains sufficiently large to preserve the superlocalization properties. The value of $r_s^{(T)}$ is obtained by an iterative process which involves discarding the smallest singular value $\sigma_i$ in (\ref{stable_c_computation}) at each iteration until condition (\ref{SLOD-stability-condition}) is satisfied.  

Before moving to the construction of HSLOD basis functions, we derive an estimate for the localization error of SLOD functions, which will be instrumental in the estimation of the localization error of HSLOD functions. To that end, first note from the definition of $\hat{\theta}^{\text{SLOD}}_{\ell,T}$ that we can write
\begin{equation*}
\hat{\theta}^{\text{SLOD}}_{\ell,T}=\mathcal{E}_{\widetilde{\omega}}\big(\mathcal{A}^{-1}_{h,\widetilde{\omega}}g^{\text{SLOD}}_{\ell,T}\big),
\end{equation*}
with $g^{\text{SLOD}}_{\ell,T}\in\mathbb{Q}^{0}(\mathcal{T}_{{\ell},\widetilde{\omega}})$. Define the counterpart of $\hat{\theta}^{\text{SLOD}}_{\ell,T}$ belonging to $\mathcal{V}_{\ell}$ as
\begin{equation}\label{SLOD-global-counterpart}
\theta_{\ell,T}=\mathcal{A}^{-1}_{h}\mathcal{E}_{\widetilde{\omega}}\big(g^{\text{SLOD}}_{\ell,T}\big).
\end{equation}
Then, based on the definitions of $\hat{\theta}^{\text{SLOD}}_{\ell,T}$, $\theta_{\ell,T}$, and Lemma \ref{Lemma-Localization-Error-bound-conormalDer}, it follows that
\begin{equation}\label{candidate-local-error-estimate}
\big\|\theta_{\ell,T}-\hat{\theta}^{\text{SLOD}}_{\ell,T}\big\|_{a}\leq \frac{1+\frac{\mathrm{diam}(\Omega)}{\pi}}{\sqrt{\alpha}} \big\|\gamma_{\hat{\theta}^{\text{SLOD}}_{\ell,T}}\big\|_{X'_{\widetilde{\omega}}}.
\end{equation}
\subsection{HSLOD basis functions at level $\ell$ leading to a Riesz stable basis}
Having computed the SLOD basis functions $\hat{\theta}^{\text{SLOD}}_{T,\ell}$ at level $\ell$, we can obtain HSLOD basis functions at the same level from their linear combination. For $\ell=0$, we simply take $\hat{\varphi}^{\text{HSLOD}}_{\ell,i}=\varphi^{\text{SLOD}}_{\ell,T_i}$ for all $i\in\{1,\ldots,N_0\}$. For $\ell>0$, we first define the set of descendants of an element $T\in \mathcal{T}_{{\ell-1}}$, obtained by refining $T$, as
\begin{equation*}
\text{ref}(T):=\{\tau\in \mathcal{T}_{{\ell}}: \tau \subset T\},
\end{equation*}
and let $\widetilde{\text{ref}}(T)$ be any subset of $\text{ref}(T)$ such that $\#\widetilde{\text{ref}}(T)=2^d-1$. Let $i\in\{1,\dots,(2^d-1)N_{\ell-1}\}$ and recall that $J_i:=\left\lfloor \frac{i-1} {(2^d-1)^{\min\{\ell,1\}}} \right \rfloor +1$. With fixed $\ell>0$ and $m\geq 1$, consider $\omega_{J_i}=\omega^{(\ell,m)}_{J_i}$, where $\omega^{(\ell,m)}_{J_i}$ is defined as in Section \ref{Sec-Construction}.  Then, define
\begin{equation}\label{HSLOD-basis-function-definition}
\hat{\varphi}^{\text{HSLOD}}_{\ell,i}=\sum_{T\in S_{\omega_{J_i}}}d_T^{(\ell,i)} \hat{\theta}_{\ell,T}^{\text{SLOD}},
\end{equation}
where $S_{\omega_{J_i}}:=\{T \in \mathcal{T}_{{\ell},\omega_{J_i}}:\hat{\theta}^{\text{SLOD}}_{T,\ell}=0 \text{ in } \Omega\setminus \omega_{J_i}\}$, and the non-trivial coefficients $(d_T)_{T\in S_{\omega_{J_i}}}$ are such that
\begin{equation}\label{Projection1-pw-for-Aorthogonality}
\Pi_{{\ell-1}}\hat{\varphi}_{\ell,i}^{\text{HSLOD}}=0.
\end{equation}
From the definition of $S_{\omega_{J_i}}$, it follows that $\hat{\varphi}_{\ell,i}^{\text{HSLOD}}$ is supported in $\omega_{J_i}$. Additionally, since the conormal derivative of $(\hat{\theta}_{\ell,T}^{\text{SLOD}})_{|_{\omega_{J_i}}}$ is small on $\Sigma_{\omega_{J_i}}=\partial \omega_{J_i}\setminus \partial \Omega$ for $T\in S_{\omega_{J_i}}$,  the conormal derivative of $(\hat{\varphi}^{\text{HSLOD}}_{\ell,i})_{|_{\omega_{J_i}}}$ will also be small on $\Sigma_{\omega_{J_i}}$ provided the coefficients $d^{(\ell,i)}_{T}$ associated with the elements $T\in S_{\omega_{J_i}}$ closest to $\Sigma_{\omega_{J_i}}$ are small enough. Further, $\big(\hat{\varphi}_{\ell,i}^{\text{HSLOD}}\big)_{|_{\omega_{J_i}}}\in \text{Ker}(\Pi_{{\ell-1},\omega_{J_i}})$, and thus condition (\ref{Projection1-pw-for-Aorthogonality}) implies the $a$-orthogonality condition (\ref{linsys-mat-form}).  
Note that $\text{dim}(S_{\omega_{J_i}})> \# \T_{{\ell-1},\omega_{J_i}}$. Hence, there are infinitely many choices of $\big(d_T^{(\ell,i)}\big)_{T\in S_{\omega_{J_i}}}$ such that (\ref{Projection1-pw-for-Aorthogonality}) holds. However, not all options will contribute to a stable HSLOD basis. To guarantee stability, we shall choose basis functions that are concentrated in different elements of the refined mesh $\mathcal{T}_{\ell}$. Therefore, define 
\begin{equation*}
\tilde{S}:=\{\tau\in \widetilde{\text{ref}}(T):T \in \mathcal{T}_{{\ell-1},\omega_{J_i}}\},
\end{equation*}
and take $\widetilde{\Pi}_{{\ell}}: H^{1}_0(\omega_{J_i})\rightarrow \mathbb{Q}^{0}(\tilde{S})$ as the $L^2$-projection onto $\mathbb{Q}^{0}(\tilde{S})$. 
Then, to obtain a stable HSLOD basis, we choose the coefficients $(d_T)_{T\in S_{\omega_{J_i}}}$ such that, in addition to (\ref{Projection1-pw-for-Aorthogonality}), $\hat{\varphi}_{\ell,i}^{\text{HSLOD}}$ satisfies
\begin{equation}\label{HSLOD-Stability-Practical-Condition}
\widetilde{\Pi}_{{\ell}}\hat{\varphi}^{\text{HSLOD}}_{\ell,i}=\chi_{T_i},
\end{equation}
where $T_i\in \widetilde{\text{ref}}(T_{J_i})$, and $T_{J_i}\in \mathcal{T}_{\ell-1,\omega_{J_i}}$ is the element around which $\omega_{J_i}$ is defined. Note that after enforcing condition (\ref{Projection1-pw-for-Aorthogonality}), we are left with $\text{dim}(\text{Ker}(\Pi_{{\ell-1},\omega_{J_i}}))=\# S_{\omega_{J_i}}-\# \T_{{\ell-1},\omega_{J_i}}$ degrees of freedom. Since $\text{dim}(\text{Ker}(\Pi_{{\ell-1},\omega_{J_i}}))\leq\#\tilde{S}$, this implies that condition (\ref{HSLOD-Stability-Practical-Condition}) cannot generally be satisfied exactly (except for the 1D case). Consequently, we select the coefficients $(d_T)_{T\in S_{\omega_{J_i}}}$ so that this condition (in its algebraic version) is satisfied in the least squares errors sense. Observe that condition \eqref{HSLOD-Stability-Practical-Condition} not only ensures that basis functions supported on the same patch are not concentrated in the exact same regions but also impels the coefficients $d^{(\ell,i)}_{T}$ to be small for $T\in S_{\omega_{J_i}}$ close to $\Sigma_{\omega_{J_i}}$.

\begin{remark}\label{Gram-Schmidt-local-orthogonalization}
To increase the degree of numerical linear independence among the $2^d-1$ HSLOD basis functions associated with a given patch $\omega_{J_i}$, after computing their $2^d-1$ corresponding coefficient vectors resulting from the application of conditions (\ref{Projection1-pw-for-Aorthogonality}) and (\ref{HSLOD-Stability-Practical-Condition}), we locally orthogonalize them (with respect to the Euclidean inner product) and take this locally orthogonal set of vectors as the new set of coefficient vectors. In practice, we observed that this additional step enhances the condition number of the HSLOD stiffness matrix.
\end{remark}

Define the normalized hierarchical basis function as
\begin{equation}\label{HSLOD-function-normalized}
\hat{\varphi}_{\ell,i}=\varphi^{\text{HSLOD}}_{\ell,i}/\|\varphi^{\text{HSLOD}}_{\ell,i}\|_a.
\end{equation}
Observe that, in general, $\textbf{span}\{\hat{\varphi}_{\ell,i}: 0\leq \ell \leq L, 1\leq i \leq \#\mathcal{B}_{\ell}\}=\hat{\mathcal{V}}_{L}\neq \mathcal{V}_{L}$. Let $\hat{d}_{T}^{(\ell,i)}=d_{T}^{(\ell,i)}/\|\varphi^{\text{HSLOD}}_{\ell,i}\|_a$ so that $\hat{\varphi}_{\ell,i}=\sum_{T\in S_{\omega}}\hat{d}_T^{(\ell,i)} \hat{\theta}_{\ell,T}^{\text{SLOD}}$. Then, with the same coefficients $\big(\hat{d}_{T}^{(\ell,i)}\big)_{T\in\mathcal{T}_{{\ell},\omega}}$ as those defining $\hat{\varphi}_{\ell,i}$, a basis of $\mathcal{V}_{L}$ can be obtained by taking
\begin{equation}\label{Counterpart-HSLOD-global-basis-func}
\varphi_{\ell,i}=\sum_{T\in S_\omega}\hat{d}_{T}^{(\ell,i)}\theta_{\ell,T}.
\end{equation}
As we will see later, these globally-supported functions are a helpful tool for evaluating how good $\hat{\mathcal{V}}_{L}$ is as an approximate space for the solution of (\ref{Continuous-Formulation}). 

Before presenting a lemma that provides a way to measure how well $\hat{\mathcal{V}}_{L}$ preserves the approximation properties of $\mathcal{V}_{L}$ as a discrete solution space for (\ref{Continuous-Formulation}), we state a result that will be useful to bound the $2$-norm of basis functions coefficient vectors in subsequent error estimates. The result follows from the Rayleigh quotient bounds of the Gram matrix of a basis.

\begin{lemma}\label{Riesz-basis}
Let $\mathcal{B}=\{b_{i}\}_{i\in\{1,\ldots,n\}}$ be a basis of an inner product space $\mathcal{V}$ with norm $\|\cdot\|_{\mathcal{V}}$ induced by the inner product $(\cdot,\cdot)_{\mathcal{V}}$. Then $\{b_{i}\}_{i\in\{1,\ldots,n\}}$ is a Riesz basis, i.e., there exists constants $0 \leq C_1 \leq C_2$ s.t. for any finite sequence of real numbers $(c_i)_{i\in\{1,\ldots,n\}}$, we have
\begin{equation*}
C_1 \sum_{i=1}^{n}|c_i|^2\leq \left\|\sum_{i=1}^{n}c_i b_{i}\right\|_{\mathcal{V}}^2\leq C_2 \sum_{i=1}^{n}|c_i|^2.
\end{equation*}
Moreover, the bounds are tight by taking $C_1$ and $C_2$ as the smallest and largest eigenvalues of the basis Gram matrix $\widetilde{\mathbf{B}}\in \mathbb{R}^{n \times n}$ s.t.\ $\widetilde{\mathbf{B}}_{ij}=(b_{i},b_{j})_{\mathcal{V}}$.
\end{lemma}

From \Cref{Riesz-basis}, and since $\|\hat{\varphi}_{\ell,i}\|_a=1$, the coefficients $\big(\hat{d}_T^{(\ell,i)}\big)_{T\in S_{\omega}}=~\hat{\mathbf{d}}^{(\ell,i)}$ defining $\hat{\varphi}_{\ell,i}$ are such that
\begin{equation*}
\|\hat{\mathbf{d}}^{(\ell,i)}\|_2\leq \frac{1}{\sqrt{\lambda_{\min}\big(\hat{\mathbb{A}}^{(\omega_{J_i})}\big)}}=\sqrt{\frac{\kappa\big(\hat{\mathbb{A}}^{(\omega_{J_i})}\big)}{\lambda_{\max}\big(\hat{\mathbb{A}}^{(\omega_{J_i})}\big)}},
\end{equation*}
where $\hat{\mathbb{A}}^{(\omega_{J_i})}\in \mathbb{R}^{n_e \times n_e}$ with $n_e=\# S_{\omega_{J_i}}$ is such that $\hat{\mathbb{A}}^{(\omega_{J_i})}_{sn}=a(\hat{\theta}^{\text{SLOD}}_{\ell,T_s},\hat{\theta}^{\text{SLOD}}_{\ell,T_n})$, with $\hat{\theta}^{\text{SLOD}}_{\ell,T_s},\hat{\theta}^{\text{SLOD}}_{\ell,T_n}$ supported in $\omega_{J_i}$. 

Furthermore, since $a(\hat{\theta}^{\text{SLOD}}_{\ell,T_s},\hat{\theta}^{\text{SLOD}}_{\ell,T_n})=1$ if $s=n$ and $a(\hat{\theta}^{\text{SLOD}}_{\ell,T_s},\hat{\theta}^{\text{SLOD}}_{\ell,T_n})\leq1$ if $s\neq n$, by construction, which implies that $1\leq \lambda_{\max}\big(\hat{\mathbb{A}}^{(\omega_{J_i})}\big)<n_e$ (using the Rayleigh quotient and Gershgorin Circle Theorem), and assuming $\delta_s$ in \eqref{SLOD-stability-condition} is small enough so that $\hat{\mathbb{A}}^{(\omega_{J_i})}$ is well-conditioned, it follows that $\|\hat{\mathbf{d}}^{(\ell,i)}\|_2$ is small.

Define
\begin{equation}\label{zeta-definition}
\zeta:=\max_{\ell\in \{0,\ldots,L\}}\max_{i\in\{1,\ldots,N^{b}_{\ell}\}} \|\hat{\mathbf{d}}^{(\ell,i)}\|_2.
\end{equation}
Also, let
\begin{equation}\label{sigma-def}
\sigma_{k,\ell}:=\max_{T_p\in S_{\omega_k}}\big\|\gamma_{\hat{\theta}^{\text{SLOD}}_{\ell,T_p}}\big\|_{X'_{\widetilde{\omega}_p}} \quad \text{ and }\quad \sigma:=\max_{\ell\in \{0,\ldots,L\}}\max_{k\in\{1,\ldots,N_{\ell-1}\}} \sigma_{k,\ell}.
\end{equation}

The following lemma presents an estimate of the localization error in the energy norm of HSLOD basis functions, giving a way to measure how good $\hat{\mathcal{V}}_{L}$ is as a substitute for $\mathcal{V}_{L}$.

\begin{lemma}\label{Lemma-error-estimate-localization-more-stable}
Let $\hat{\varphi}_{\ell,i}$ and $\varphi_{\ell,i}$ be the basis functions of $\hat{\mathcal{V}}_{L}$ and $\mathcal{V}_{L}$, respectively, defined in (\ref{HSLOD-function-normalized}) and (\ref{Counterpart-HSLOD-global-basis-func}). Then the following estimate holds.
\begin{equation*}
\|\varphi_{i,\ell}-\hat{\varphi}_{i,\ell}\|_{a}\leq \frac{\big(1+\frac{\mathrm{diam}(\Omega)}{\pi}\big)\sqrt{N_E}}{\sqrt{\alpha}}\zeta \sigma,
\end{equation*}
where $\sigma$ is given in (\ref{sigma-def}), $\zeta$ in (\ref{zeta-definition}), and $N_E$  is the largest number of elements that can possibly be contained within the supporting patches of basis functions.
\end{lemma}
\begin{proof}
From (\ref{HSLOD-function-normalized}), (\ref{Counterpart-HSLOD-global-basis-func}), (\ref{candidate-local-error-estimate}), (\ref{sigma-def}), (\ref{zeta-definition}), the Cauchy-Schwarz inequality, and letting $\omega_{J_i}=\omega_{J_i}^{(\ell,m)}$ with $J_i=\left\lfloor \frac{i-1} {(2^d-1)^{\min\{\ell,1\}}} \right \rfloor +1$, it follows that
\begin{eqnarray}\label{Localization-error-more-stable-basis}
\nonumber \|\varphi_{i,\ell}-\hat{\varphi}_{i,\ell}\|_{a}&=&\Big\|\sum_{T\in S_{\omega_{J_i}}}\hat{d}_T\left(\theta_{\ell,T}-\hat{\theta}^{\text{SLOD}}_{\ell,T}\right)\Big\|_{a}\\
\nonumber &\leq & \max_{T\in S_{\omega_{J_i}} } \left\{\big\|\theta_{\ell,T}-\hat{\theta}^{\text{SLOD}}_{\ell,T}\big\|_{a}\right\} \sum_{T\in S_{\omega_{J_i}} }|\hat{d}_T|\\
\nonumber &\leq & \frac{1+\frac{\mathrm{diam}(\Omega)}{\pi}}{\sqrt{\alpha}}\max_{T\in S_{\omega_{J_i}}   } \left\{\big\|\gamma_{\hat{\theta}^{\text{SLOD}}_{\ell,T}}\big\|_{X'_{\omega_{J_i}}}\right\} \sqrt{|S_{\omega_{J_i}}|}\Big(\sum_{T\in S_{\omega_{J_i}}}\hat{d}_T^2\Big)^{\frac{1}{2}}\\
&\leq & \frac{\big(1+\frac{\mathrm{diam}(\Omega)}{\pi}\big)\sqrt{N_E}}{\sqrt{\alpha}}\zeta \sigma.
\end{eqnarray}
\end{proof}

{
\begin{remark}
Since non-trivial corrections are added to LOD basis functions so that the localization error of SLOD basis functions reduces with respect to the LOD counterpart, we have $\sigma\leq \sigma^{\text{LOD}}$, where $\sigma^{\text{LOD}}$ is the value of $\sigma$ in the LOD case. Note that $\sigma^{\text{LOD}}\leq C_{\dagger} H^{-1}_{L} e^{-C m}$ (see \cite{HaPe21b,AHP21,MaP14}), where $m$ is the patch order and $C_{\dagger}$ is a constant independent of the mesh size and $m$. Therefore, it holds that $\sigma\leq C_{\dagger} H^{-1}_{L} e^{-C m}$. As we will observe in \Cref{sec:Numerical-Experiments}, this is generally a very pessimistic bound for $\sigma$ when the non-trivial corrections provided in \Cref{Stable-SLOD-subsection} are applied.
\end{remark}
}

\begin{remark}
In the remainder of the paper, we will refer to the basis functions of $\hat{\mathcal{V}}_{L}$ and $\mathcal{V}_{L}$ by their notation with global indices. Thus, we have $\hat{\mathcal{V}}_{L}= \text{span}\{\hat{\varphi}_i\}_{i\in\{1,\ldots,N_{L}\}}$, $\mathcal{V}_{L}= \text{span}\{\varphi_i\}_{i\in\{1,\ldots,N_{L}\}}$ , where, with $i=N_{\ell-1}+j$,
\begin{equation}\label{basis-global-notation}
\hat{\varphi}_{i}=\hat{\varphi}_{\ell,j},\text{ and }\varphi_{i}=\varphi_{\ell,j}.
\end{equation}
Of course, the functions $\hat{\varphi}_{i}$ depend on the parameters $m_i$ defining the size of their supporting patches. However, to simplify notation, we do not include these parameters as subscripts in (\ref{basis-global-notation}).
\end{remark}

\subsection{Behavior of HSLOD basis functions at a given level and their associated stiffness matrix}\label{Subsec:HSLOD-Behavior}
In this subsection, we study the behavior of the normalized HSLOD basis functions constructed in the previous subsection. To that end, we first introduce a theorem (inspired by ideas from \cite[Lemma 6]{FeP20}) that provides an estimate for the condition number of the diagonal blocks of the stiffness matrix associated with the normalized HSLOD basis, and that reveals the mesh independence of these condition numbers, for which condition \eqref{Projection1-pw-for-Aorthogonality} plays a central role. Furthermore, this theorem will be useful to assess the quality of the compression operations presented in \Cref{Sec-Compressed-Op-construction-and-Error-Analysis}. After presenting the theorem, we shall explore how these condition numbers depend on the degree of linear independence (at each level) of basis functions and the contrast of the diffusion coefficient. From this study, we shall determine a quantity that can be used as a measure of how well-behaved the hierarchical basis is.

Consider $\hat{\mathcal{V}}_{L}=\text{span}\{\hat{\varphi}_{i}\}_{i\in\{1,\ldots,N_L\}}$, where $\hat{\varphi}_i$ is defined in (\ref{basis-global-notation}), and define $\hat{\mathbb{A}}_{H_L}\in \mathbb{R}^{N_L \times N_L}$ by $\big(\hat{\mathbb{A}}_{H_{L}}\big)_{ij}=a(\hat{\varphi}_i,\hat{\varphi}_j)$. Further, let $\hat{\mathbb{A}}_{\mathbf{\ell \ell}}^{(H_L)}\in \mathbb{R}^{N^b_{\ell} \times N^b_{\ell}}$ be the $\ell$-th diagonal block of $\hat{\mathbb{A}}_{H_L}$ such that $\big(\hat{\mathbb{A}}_{\mathbf{\ell \ell}}^{(H_L)}\big)_{ij}=a(\hat{\varphi}_{\ell,i},\hat{\varphi}_{\ell,j})$ with $\hat{\varphi}_{\ell,i},\hat{\varphi}_{\ell,j}\in \hat{\mathcal{B}}_{\ell}$.

Let $\Pi_{\ell}\hat{\varphi}_{\ell,i}=\sum_{T\in \mathcal{T}_{\ell}}p_{T}^{(i)}\chi_{T}$ and define $\mathbf{P}\in \mathbb{R}^{N_{\ell} \times N_{\ell}}$ such that $\mathbf{P}_{ij}=p^{(j)}_{T_i}$. From \eqref{Projection1-pw-for-Aorthogonality}, \eqref{HSLOD-Stability-Practical-Condition}, and \eqref{HSLOD-function-normalized}, it follows that $\mathbf{P}=\widetilde{\mathbf{P}}\mathbf{N}$, where $\widetilde{\mathbf{P}}\in \mathbb{R}^{N_{\ell}\times N_{\ell} }$ is such that $\widetilde{\mathbf{P}}_{ij}=\tilde{p}_{T_i}^{(j)}$, with $\Pi_{\ell}\hat{\varphi}^{\text{HSLOD}}_{\ell,j}=\sum_{T\in\mathcal{T}_{\ell}}\tilde{p}_{T}^{(j)}\chi_T$, and $\mathbf{N}^{N_{\ell}\times N_{\ell} }$ is the diagonal matrix such that $\mathbf{N}_{ii}=\frac{1}{\|\hat{\varphi}^{\text{HSLOD}}_{\ell,i}\|_a}$. In \Cref{Appendix-Eigen-Estimate-for-mesh-indep}, it is shown that 
\begin{equation}\label{PTP-lambda-min-bound}
\lambda_{\mathrm{min}}^{-1}(\mathbf{P}^T\mathbf{P})\leq \frac{C}{\lambda_{\mathrm{min}}(\widetilde{\mathbf{P}}^T \widetilde{\mathbf{P}})}H_{\ell}^{d-2},
\end{equation}
where $C$ and $\lambda_{\mathrm{min}}(\widetilde{\mathbf{P}}^T \widetilde{\mathbf{P}})$ are mesh independent quantities, and $C$ may depend on $\beta$ and $\alpha$. With $\hat{\mathbf{P}}=H_{\ell}^{\frac{d}{2}-1}\mathbf{P}$, it follows from \eqref{PTP-lambda-min-bound} that $\lambda_{\mathrm{min}}(\hat{\mathbf{P}}^T \hat{\mathbf{P}})$ has a mesh independent lower bound. Then, the diagonal blocks of the normalized-HSLOD stiffness matrix possess the following properties.
\begin{theorem}\label{HSLOD-Stiffness-Blocks-condnum-estimate}
The condition number of the diagonal block $\hat{\mathbb{A}}^{H_L}_{\ell \ell}$ of the stiffness matrix $\hat{\mathbb{A}}_{H_L}$ associated with the hierarchical basis $\hat{\mathcal{V}}_{L}$ is mesh independent for $\ell>0$ and $\mathcal{O}(H_0^{-2})$ for $\ell=0$. 

Furthermore, the smallest eigenvalue of $\hat{\mathbb{A}}_{\mathbf{\ell \ell}}^{(H_L)}$ has the lower bound
\begin{equation}\label{min_eig_HSLOD_block}
\lambda_{\mathrm{min}}\big(\hat{\mathbb{A}}_{\mathbf{\ell \ell}}^{(H_L)}\big)\geq \begin{cases} 
\frac{\pi^2\alpha}{4}\lambda_{\mathrm{min}}(\hat{\mathbf{P}}^T\hat{\mathbf{P}})& \ell>0,\\
\frac{\pi^2\alpha}{\mathrm{diam}^2(\Omega)}H^{2}_{\ell}\lambda_{\mathrm{min}}(\hat{\mathbf{P}}^T\hat{\mathbf{P}})& \ell=0,\\
\end{cases}
\end{equation}
where $\alpha>0$ is given in \eqref{A-spectral-bound} and $\lambda^{-1}_{\mathrm{min}}(\hat{\mathbf{P}}^T\hat{\mathbf{P}})=\mathcal{O}(1)$, and
\begin{equation}\label{HSLOD-block-condnum-bound}
\kappa\big(\hat{\mathbb{A}}^{H_L}_{\ell \ell} \big)\leq 
\begin{cases}
\frac{4 \mathfrak{n}_{o,\ell}}{\pi^2 \alpha \lambda_{\mathrm{min}}(\hat{\mathbf{P}}^T\hat{\mathbf{P}})}& \ell>0,\\
\frac{\mathrm{diam}^2(\Omega) \mathfrak{n}_{o,\ell}}{\pi^2 \alpha  \lambda_{\mathrm{min}}(\hat{\mathbf{P}}^T\hat{\mathbf{P}})}H^{-2}_{\ell}& \ell=0,
\end{cases}
\end{equation}
where $\mathfrak{n}_{o,\ell}$ is the maximum possible number of functions $\hat{\varphi}_{i,\ell}\in \hat{\mathcal{B}_{\ell}}$  whose supports overlap over a region of $\Omega$. 
\end{theorem}
\begin{proof}
First, note that using \Cref{Riesz-basis} and for an arbitrary $\mathbf{c}=\left(c_i\right)_{i=1}^{N_{\ell}}\in \mathbb{R}^{N_{\ell}}$ we have
\begin{equation}\label{L2norm-lower-bound-lincomb-L2projections}
\Big\|\sum_{i=1}^{N_{\ell}}c_i\Pi_{\ell}\hat{\varphi}_{\ell,i}\Big\|_{L^2(\Omega)}^2=H_{\ell}^d\mathbf{c}^T\mathbf{P}^T\mathbf{P} \mathbf{c}\geq H_{\ell}^d \lambda_{\mathrm{min}}(\mathbf{P}^T\mathbf{P})\|\mathbf{c}\|_2^2.
\end{equation}
Note also that the orthogonal $L^2$-projection operator $\Pi_{\ell}$ has the following two properties.
\begin{equation}\label{L2-proj-reduced-norm-condition}
\|\Pi_{\ell}v\|_{L^2(\Omega)}\leq\|v\|_{L^2(\Omega)}\quad\text{for all }v\in L^2(\Omega),\text{ and}
\end{equation}
\begin{equation}\label{L2-complementary-projection-L2norm-condition}
\|(1-\Pi_{\ell})v\|_{L^2(\Omega)}\leq \frac{H_{\ell}}{\pi}\|\nabla v\|_{L^2(\Omega)}\quad\text{for all }v\in H^1(\Omega).
\end{equation}
Then, from \eqref{L2norm-lower-bound-lincomb-L2projections}, \eqref{L2-proj-reduced-norm-condition}, \eqref{Projection1-pw-for-Aorthogonality}, \eqref{L2-complementary-projection-L2norm-condition}, and since $H_{\ell-1}=2 H_{\ell}$, we obtain
\begin{eqnarray}\label{blocks-condnum-derivation}
\nonumber H^d_{\ell}\lambda_{\mathrm{min}}(\mathbf{P}^T\mathbf{P})\sum_{i=1}^{N_{\ell}}c_{i}^{2}&\leq & \Big\|\sum_{i=1}^{N_{\ell}}c_i\Pi_{\ell}\hat{\varphi}_{\ell,i}\Big\|_{L^2(\Omega)}^2
\leq \Big\|\sum_{i=1}^{N_{\ell}}c_i\hat{\varphi}_{\ell,i}\Big\|_{L^2(\Omega)}^2
=\Big\|(1-\Pi_{\ell-1})\sum_{i=1}^{N_{\ell}}c_i\hat{\varphi}_{\ell,i}\Big\|_{L^2(\Omega)}^2\\ 
&\leq & \frac{H_{\ell-1}^2}{\pi^2\alpha}\Big\|\sum_{i=1}^{N_{\ell}}c_i\hat{\varphi}_{\ell,i}\Big\|_{a}^2
=\frac{4 H_{\ell}^2}{\pi^2\alpha}\Big\|\sum_{i=1}^{N_{\ell}}c_i\hat{\varphi}_{\ell,i}\Big\|_{a}^2.
\end{eqnarray}
Thus, using \Cref{Riesz-basis} and the definition of $\hat{\mathbf{P}}$, we have that \eqref{blocks-condnum-derivation} implies the first line of \eqref{min_eig_HSLOD_block}. The second line of \eqref{min_eig_HSLOD_block} is obtained from the first two inequalities in the first line of \eqref{blocks-condnum-derivation}, the Poincar\'e-Friedrichs inequality, \Cref{Riesz-basis}, and the definition of $\hat{\mathbf{P}}$.

Since $|a(\hat{\varphi}_{\ell,i},\hat{\varphi}_{\ell,j})|\leq 1$ for all $\hat{\varphi}_{\ell,i},\hat{\varphi}_{\ell,j}\in \hat{\mathcal{B}}_{\ell}$, using the Gershgorin Circle Theorem we obtain $\lambda_{\mathrm{max}}\leq \mathfrak{n}_{o,\ell}$. Then, the estimate \eqref{HSLOD-block-condnum-bound} follows.
\end{proof}

Using the Rayleigh quotient, it can be shown that (see \Cref{Appendix-Eigen-Estimate-for-mesh-indep})
\begin{equation}
\lambda_{\mathrm{min}}(\hat{\mathbf{P}}^T\hat{\mathbf{P}})=H^{d-2}_{\ell}\lambda_{\mathrm{min}}(\mathbf{P}^T\mathbf{P})\geq H^{d-2}_{\ell}\lambda_{\mathrm{min}}(\widetilde{\mathbf{P}}^T\widetilde{\mathbf{P}}) \lambda_{\mathrm{min}}(\mathbf{N}^2).
\end{equation}
Note that $\lambda_{\mathrm{min}}(\widetilde{\mathbf{P}}^T\widetilde{\mathbf{P}})=\mathcal{O}(1)$, i.e., it is mesh independent (see \Cref{Appendix-Eigen-Estimate-for-mesh-indep}). The $i$-th column entries of $\widetilde{\mathbf{P}}$ indicate the degree of concentration of $\hat{\varphi}_{\ell,i}$ in each element of $\mathcal{T}_{\ell}$. Thus, $\lambda_{\mathrm{min}}(\widetilde{\mathbf{P}}^T\widetilde{\mathbf{P}})$ is a measure of the degree of linear independence of the set of basis functions $\{\hat{\varphi}_{\ell,i}\}_{i=1}^{N^b_{\ell}}$ (and thus of the basis quality). The farther from $0$ it is the more independent the basis is. The quantity $H^{d-2}_{\ell}\lambda_{\mathrm{min}}(\mathbf{N}^2)=\mathcal{O}(1)$ (see \Cref{Appendix-Eigen-Estimate-for-mesh-indep}) serves as a measure of the effect of the contrast on the basis functions. Thus, the above inequality together with \eqref{HSLOD-block-condnum-bound} suggest that the condition number of $\hat{\mathbb{A}}^{H_L}_{\ell \ell}$ may be affected by  the degree of linear independence of the basis functions (at level $\ell$) and the contrast of the diffusion coefficient.

\begin{remark}\label{Remark-basis-well-behavior-measure}
If conditions \eqref{Projection1-pw-for-Aorthogonality} and \eqref{HSLOD-Stability-Practical-Condition} are satisfied exactly, $\lambda_{\mathrm{min}}(\widetilde{\mathbf{P}}^T\widetilde{\mathbf{P}})=1$ (cf. \Cref{Appendix-Eigen-Estimate-for-mesh-indep}). In practice, however, we cannot generally satisfy condition \eqref{HSLOD-Stability-Practical-Condition} exactly, and we do it in the least-squares-error sense. Consequently, if this least-squares error is small enough, the resulting hierarchical basis should be well-behaved. 
\end{remark}
\begin{remark}
If the local orthogonalization procedure mentioned in \Cref{Gram-Schmidt-local-orthogonalization} is performed via the QR factorization method, the normalization step in that algorithm will make $\lambda_{\mathrm{min}}(\widetilde{\mathbf{P}}^T\widetilde{\mathbf{P}})$ scale with $1/H^{d-2}_{\ell}$. In that case, it follows that we could use $H^{d-2}_{\ell}\lambda_{\mathrm{min}}(\widetilde{\mathbf{P}}^T\widetilde{\mathbf{P}})$ as a measure of how well-behaved the hierarchical basis would be.
\end{remark}

With our practical hierarchical $a$-orthogonal basis already defined, we are ready to discuss the construction of a compressed operator that approximates $\mathcal{A}^{-1}_h$ (and therefore $\mathcal{A}^{-1}$).

\section{Operator compression and inversion}\label{Sec-Compressed-Op-construction-and-Error-Analysis}
In the introduction, we mentioned that the finite-rank operator $\mathcal{A}_h$ is a good approximation of the infinite-rank operator $\mathcal{A}^{-1}$ in the sense that $\|\mathcal{A}^{-1}f-\mathcal{A}^{-1}_h f\|_a$ is small for all $f\in L^2(\Omega)$. In this section, we present a sparse-compressed approximation of the operator $\mathcal{A}^{-1}$ obtained by approximating $\mathcal{A}^{-1}_h$ by a sparse-compressed finite-rank operator $\mathcal{S}$ of the form $\mathcal{S}=\mathcal{L}\circ\mathfrak{S}\circ\mathcal{R}$. Here $\mathcal{L},\mathfrak{S},\mathcal{R}$ are linear transformations, and $\mathfrak{S}:\mathbb{R}^{N_L}\rightarrow \mathbb{R}^{N_L}$ is such that $\mathfrak{S}(\mathbf{x})=\mathbb{S}\mathbf{x}$, with $\mathbb{S}\in \mathbb{R}^{N_L \times N_L}$. The term `sparse-compressed' in this context means that $\textbf{rank}(\mathcal{S})\ll \textbf{rank}\left(\mathcal{A}^{-1}_{h}\right)$ and $\mathbb{S}$ is a sparse matrix. The operator $\mathcal{S}$ is obtained after a number of compression operations, which are the subject of our following discussion. We present four possible compression operations. It is worth mentioning that the number of compression operations used in practice depends upon how well-conditioned the diagonal blocks of the stiffness matrix associated with the localized hierarchical basis of $\hat{\mathcal{V}}_{L}$ are.

\begin{remark}
If the condition numbers of the diagonal blocks of the stiffness matrix associated with the localized hierarchical basis of $\hat{\mathcal{V}}_{L}$ are small, then the sparse-compressed operator $\mathcal{S}$ can actually be computed. In that case, for a given $f\in L^2(\Omega)$, the computation of the approximate solution of \eqref{Continuous-Formulation} using the computed sparse-compressed operator $\mathcal{S}$ reduces (in its algebraic form) to the computation of a matrix-vector product.
\end{remark}

\begin{remark}
A different approach to obtain a sparse-compressed approximation of the solution operator $\mathcal{A}^{-1}$ is given in \cite{SchKatOwh21}, where the stiffness matrix associated with the (globally-supported) gamblets is approximated via a sparse Cholesky factorization of it.
\end{remark}

\subsection{First compression stage: rank reduction and sparsification by using the superlocalized basis}
Recall that $\hat{\mathcal{V}}_{L}=\text{span}\{\hat{\varphi}_{i}\}_{i\in\{1,\ldots,N_L\}}$, where $\hat{\varphi}_i$ is defined in (\ref{basis-global-notation}), and $\hat{\mathbb{A}}_{H_L}\in \mathbb{R}^{N_L \times N_L}$ is given by $\big(\hat{\mathbb{A}}_{H_{L}}\big)_{ij}=a(\hat{\varphi}_i,\hat{\varphi}_j)$. Let the operator $\mathcal{R}:L^2(\Omega)\rightarrow \mathbb{R}^{N_L}$ be such that
\begin{equation}\label{Right-Lin-Tr}
\mathcal{R}f=\left[(f,\hat{\varphi}_{1})_{L^2(\Omega)},\ldots,(f,\hat{\varphi}_{N_L})_{L^2(\Omega)}\right]^T:=\mathbf{f},
\end{equation}
and $\mathcal{L}:\mathbb{R}^{N_L}\rightarrow \hat{\mathcal{V}}_{L}$ be given by
\begin{equation*}
\mathcal{L}(\mathbf{c})=\sum_{i=1}^{N_L}c_i \hat{\varphi}_{i},
\end{equation*}
where $\mathbf{c}=\left(c_i\right)_{i=1}^{N_L}$. The operator $\hat{\mathfrak{S}}^{-1}:\mathbb{R}^{N_L}\rightarrow \mathbb{R}^{N_L}$ is such that $\hat{\mathfrak{S}}^{-1}(\mathbf{x})=\hat{\mathbb{A}}^{-1}_{H_L} \mathbf{x}$.  We define the operator $\hat{\mathcal{S}}:L^2(\Omega)\rightarrow \hat{\mathcal{V}}_{L}$ by 
\begin{equation}\label{S-hat-operator-def}
\hat{\mathcal{S}}=\mathcal{L}\circ\hat{\mathfrak{S}}^{-1}\circ\mathcal{R}.
\end{equation}
Note that $\mathbf{\text{rank}}(\hat{\mathcal{S}})=\mathbf{\text{rank}}(\mathbb{A}_{H_L})=N_L$. We have the following approximation estimate. 

\begin{theorem}
Let $\mathcal{A}^{-1}:L^2(\Omega)\rightarrow H^{1}_{0}(\Omega)$ be the solution operator of \eqref{Continuous-Formulation}, $\mathcal{A}^{-1}_h:L^2(\Omega)\rightarrow V_h$ the operator defined by \eqref{Fine-problem-def}, and $\hat{\mathcal{S}}:L^2(\Omega)\rightarrow \hat{\mathcal{V}}_{L}$ the operator given by (\ref{S-hat-operator-def}). Then, the following approximation estimate holds. 
\begin{eqnarray}\label{S-hat-Error-Estimate}
\nonumber \|\mathcal{A}^{-1}f-\hat{\mathcal{S}}f\|_a\leq \|\mathcal{A}^{-1}f-\mathcal{A}_{h}^{-1}f\|_a &+& \frac{H_L}{\pi \sqrt{\alpha}} \left\|f-\Pi_{H_L}f\right\|_{L^2(\Omega)}\\ &+& \frac{\sqrt{N_E} \big(1+\frac{\mathrm{diam}(\Omega)}{\pi}\big)}{\sqrt{\alpha}} \zeta \sigma \frac{\left\|f\right\|_{L^2(\Omega)}}{\sqrt{\lambda_{\mathrm{min}}(\mathbf{G})}},
\end{eqnarray}
where  $\zeta$ is given in \eqref{zeta-definition}, $\sigma$ in \eqref{sigma-def}, $N_E$ is the largest number of elements that can possibly be contained within the supporting patches of basis functions, and $\mathbf{G}\in \mathbb{R}^{N_L\times N_L}$ is such that \ $\mathbf{G}_{ij}=(g_{i},g_{j})_{L^2(\Omega)}$, with $\{g_i\}_{i\in\{1,\ldots,N_L\}}\subset \mathbb{Q}^{0}(\mathcal{T}_{L})$ being the basis companion of $\{\varphi_i\}_{i\in\{1,\ldots,N_L\}}$~.
\end{theorem}
\begin{proof}
Let $f\in L^2(\Omega)$, $u=\mathcal{A}^{-1}f$, $u_h=\mathcal{A}^{-1}_{h}f$, $\tilde{u}=\mathcal{A}^{-1}_{h}\Pi_{L}f$, and $\hat{u}=\hat{\mathcal{S}}f$. From the definition of $\hat{\mathcal{S}}$, we have that
\begin{equation*}
a(\hat{u},v)=(f,v)_{L^2(\Omega)} \;\;\;\; \text{for all } v\in \hat{\mathcal{V}}_{L}. 
\end{equation*}
C\'ea's lemma establishes that
\begin{equation*}
\|u-\hat{u}\|_a\leq \|u-\hat{w}\|_a \;\;\;\;\;\;\text{for all } \hat{w}\in \hat{\mathcal{V}}_{L}.
\end{equation*}
Thus, using C\'ea's lemma and the triangle inequality, we obtain
\begin{equation}\label{S-hat-TriangIneq-bound}
\|u-\hat{u}\|_a \leq \|u-u_h\|_a+\|u_h-\tilde{u}\|_a+\|\tilde{u}-\hat{w}\|_a,
\end{equation}
where $\hat{w}\in \hat{\mathcal{V}}_{L}$ is arbitrary. To obtain the error estimate, we derive bounds for the last two terms on the r.h.s.\ of the above inequality. 

From (\ref{Continuous-Formulation}), the definition of $\tilde{u}$, the Cauchy-Schwarz inequality, (\ref{A-spectral-bound}), the property of $\Pi_{L}$ given by \eqref{L2-complementary-projection-L2norm-condition}, and noticing that $u_h-\tilde{u}\in V_h \subset H^{1}_{0}(\Omega)$, we have
\begin{eqnarray}\label{S-hat-error-bound-rhs1}
\nonumber \|u_h-\tilde{u}\|^{2}_{a}=a(u_h-\tilde{u},u_h-\tilde{u})&=&a(u_h,u_h-\tilde{u})-a(\tilde{u},u_h-\tilde{u})\\
\nonumber &=&(f,u_h-\tilde{u})_{L^2(\Omega)}-(\Pi_{L}f,u_h-\tilde{u})_{L^2(\Omega)}\\
\nonumber &=&(f-\Pi_{L}f,u_h-\tilde{u})_{L^2(\Omega)}\\
\nonumber &=&(f-\Pi_{L}f,u_h-\tilde{u}-\Pi_{L}(u_h-\tilde{u}))_{L^2(\Omega)}\\
\nonumber &\leq & \|f-\Pi_{L}f\|_{L^2(\Omega)}\|u_h-\tilde{u}-\Pi_{L}(u_h-\tilde{u})\|_{L^2(\Omega)}\\
&\leq & \frac{H_L}{\pi \sqrt{\alpha}} \left\|f-\Pi_{L}f\right\|_{L^2(\Omega)}\|u_h-\tilde{u}\|_{a}.
\end{eqnarray}
Let $\Pi_{L}f=\sum_{i=1}^{N_L}c_i g_i$. From the definition of $\tilde{u}$, and noticing that $\varphi_i=\mathcal{A}_{h}^{-1}g_i$ (since $g_i$ is the $\mathbb{Q}^0$-companion of $\varphi_i$), we obtain $\tilde{u}=\sum_{i=1}^{N_L}c_i \varphi_i$, where $\varphi_i$ is defined in (\ref{basis-global-notation}). With $\hat{w}=\sum_{i=1}^{N_L}c_i \hat{\varphi}_i$, $\Omega_i=\mathrm{supp}(\hat{\varphi}_i)$, the Cauchy-Schwarz inequality, \Cref{Lemma-Localization-Error-bound-conormalDer}, \eqref{sigma-def}, Lemma \ref{Riesz-basis}, and since $\|\Pi_{L}f\|_{L^2(\Omega)}\leq \|f\|_{L^2(\Omega)}$, we can bound the second r.h.s.\ term of \eqref{S-hat-TriangIneq-bound} as follows:
\begin{eqnarray}\label{S-hat-error-bound-rhs2}
\nonumber \|\tilde{u}-\hat{w}\|_{a}^2&=& \sum_{i=1}^{N_L}c_i a\left(\varphi_{i}-\hat{\varphi}_{i},\tilde{u}-\hat{w}\right) \\
\nonumber &\leq & \left(\sum_{i=1}^{N_L}c_i^2\right)^{\frac{1}{2}} \left(\sum_{i=1}^{N_L}a\left(\varphi_{i}-\hat{\varphi}_{i},\tilde{u}-\hat{w}\right)^2 \right)^{\frac{1}{2}} \\
\nonumber &\leq & \frac{\left(1+\frac{\mathrm{diam}(\Omega)}{\pi}\right)}{\sqrt{\alpha}}\zeta \sigma \left(\sum_{i=1}^{N_L}\|(\tilde{u}-\hat{w})_{|_{\omega_i}}\|^2_{a_{\Omega_i}}\right)^{\frac{1}{2}}\left(\sum_{i=1}^{N_L}c_i^2\right)^{\frac{1}{2}}\\
&\leq & \frac{\left(1+\frac{\mathrm{diam}(\Omega)}{\pi}\right)}{\sqrt{\alpha}} \zeta \sigma \left(N_E\|(\tilde{u}-\hat{w}\|^2_{a}\right)^{\frac{1}{2}}\left(\frac{\left\|f\right\|_{L^2(\Omega)}}{\sqrt{\lambda_{\mathrm{min}}(\mathbf{G})}}\right)\cdot
\end{eqnarray}
Thus, from (\ref{S-hat-TriangIneq-bound}), (\ref{S-hat-error-bound-rhs1}), and (\ref{S-hat-error-bound-rhs2}), the estimate in (\ref{S-hat-Error-Estimate}) follows.
\end{proof}
\begin{remark}
The third term of the r.h.s.\ of \eqref{S-hat-Error-Estimate} suggests that the error of the approximate solution due to the basis localization procedure will be small provided that $\sigma$ is small and $\lambda_{\mathrm{min}}(\mathbf{G})$ is large enough. In all our numerical experiments, we observed that there exist a set $\{g_i\}_{i\in\{1,\ldots,N_L\}}\subset \mathbb{Q}^{0}(\mathcal{T}_{L})$ as defined in \Cref{S-hat-Error-Estimate} such that these two sufficient conditions for the smallness of the error due to basis localization are satisfied, even in the presence of high-contrast channels in the diffusion coefficient. Note in particular that those two sufficient conditions are satisfied whenever the set $\{g_i\}_{i\in\{1,\ldots,N_L\}}\subset \mathbb{Q}^{0}(\mathcal{T}_{L})$ is a Riesz stable basis with a corresponding small $\sigma$ (in accordance with Assumption 5.2 in \cite{HaPe21b}). Further, note that the error due to basis localization will exhibit a superexponential decay if $\frac{\zeta \sigma}{\sqrt{\lambda_{\mathrm{min}}(\mathbf{G})}}$ does so.
\end{remark}

\subsection{Second compression stage: sparsification by discarding off\hyph block\hyph diagonal entries of $\hat{\mathbb{A}}_{H_L}$}\label{sec:second compression stage}
If we have a hierarchical basis where all the basis functions pertaining to different levels are $a$-orthogonal, the associated stiffness matrix $\hat{\mathbb{A}}_{H_L}$ will be block-diagonal. As mentioned earlier, in order to obtain an approximation space with a localized hierarchical basis (which can be computed efficiently), we might loose $a$-orthogonality among some basis functions at different levels. This results in a stiffness matrix $\hat{\mathbb{A}}_{H_L}$ that is no longer block-diagonal, since some of its off-block-diagonal entries containing the inner product of non-$a$-orthogonal functions from different levels will be non-zero (but small).  

Let $\check{\mathbb{A}}_{H_L}\in \mathbb{R}^{N_L \times N_L}$ be the block-diagonal matrix such that 
\begin{equation*}
\left(\check{\mathbb{A}}_{H_L}\right)_{ij}=\begin{cases}
a(\hat{\varphi}_{i},\hat{\varphi}_{j}) & \text{if level}(\hat{\varphi}_{i})=\text{level}(\hat{\varphi}_{j}),\\
0& \text{otherwise,}
\end{cases} 
\end{equation*}
and $\check{\mathfrak{S}}^{-1}_{H_L}:\mathbb{R}^{N_L} \rightarrow \mathbb{R}^{N_L}$ be given by $\check{\mathfrak{S}}^{-1}_{H_L}(\mathbf{x})=\check{\mathbb{A}}^{-1}_{H_L}\mathbf{x}$.
Then, we define the operator $\check{S}:L^2(\Omega) \rightarrow \hat{\mathcal{V}}_{{L}}$ by
\begin{equation}\label{S-check-def}
\check{S}=\mathcal{L}\circ\check{\mathfrak{S}}_{H_{L}}^{-1}\circ\mathcal{R}.
\end{equation}
To estimate how well $\hat{\mathcal{S}}$ is approximated by the more compressed operator $\check{\mathcal{S}}$, we need to determine first how well $\check{\mathbb{A}}_{H_L}^{-1}$ approximates $\hat{\mathbb{A}}_{H_L}^{-1}$. The following lemma provides a way to compare these matrices via the quantification of the change in the solution of a linear system when the coefficients matrix is perturbed by throwing away all off-block-diagonal entries.

\begin{lemma}\label{lem:discrading-off-block-entries}
Let $\mathbf{A}\in \mathbb{R}^{n \times n}$ be a SPD matrix and $\bar{\mathbf{A}}\in \mathbb{R}^{n \times n}$ the block-diagonal matrix whose block-diagonal entries coincide with those of $\mathbf{A}$. Let $\mathcal{N}_b\in \mathbb{N}$ be the number of diagonal blocks of $\bar{\mathbf{A}}$, $n_i\in\mathbb{N}$ the number of rows of its $i$-th diagonal block, $\delta:=\mathbf{A}-\bar{\mathbf{A}}$, and $\delta_{\mathbf{i}}$ the rectangular submatrix of $\delta$ obtained by collecting the rows of $\delta$ whose indices coincide with those of the rows of $\mathbf{A}$ containing its $i$-th block.  Let $\mathbf{b} \in \mathbb{R}^n$ be arbitrary, $\mathbf{x}\in\mathbb{R}^n$ s.t.\ $\mathbf{A}\mathbf{x}=\mathbf{b}$, and $\bar{\mathbf{x}}\in\mathbb{R}^n$ s.t.\ $\bar{\mathbf{A}}\bar{\mathbf{x}}=\mathbf{b}$. Then the following error bound holds.
\begin{equation}\label{Error-Estimate-Perturbed-LinSys}
\|\mathbf{x}-\bar{\mathbf{x}}\|_2\leq \Big( \sum_{i=1}^{\mathcal{N}_\mathbf{b}}\lambda_{\mathrm{min}}^{-2}\left(\mathbf{A_{ii}}\right)\|\delta_{\mathbf{i}}\|_2^2 \|\mathbf{x}\|_2^2 \Big)^{\frac{1}{2}}
\end{equation}
where $\mathbf{A_{ii}}\in \mathbb{R}^{n_{i} \times n_i}$ is the $i$-th diagonal block of $\bar{\mathbf{A}}$ and $\mathbf{A}$. 
\end{lemma} 

\begin{proof}
Let $\mathbf{x}=[\mathbf{x}_1^T,\ldots,\mathbf{x}_{\mathcal{N}_{\mathbf{b}}}^T]^T$, $\bar{\mathbf{x}}=[\bar{\mathbf{x}}_1^T,\ldots, \bar{\mathbf{x}}_{\mathcal{N}_{\mathbf{b}}}^T]^T$, and $\mathbf{b}=[\mathbf{b}_1^T,\ldots,\mathbf{b}_{\mathcal{N}_\mathbf{b}}^T]^T$, where $\mathbf{x}_{\mathbf{j}},\bar{\mathbf{x}}_{\mathbf{j}},\mathbf{b}_{\mathbf{j}}\in \mathbb{R}^{n_j}$. Since the diagonal blocks of $\bar{\mathbf{A}}$ and $\mathbf{A}$ coincide, the diagonal blocks of $\delta$ are zero, i.e., $\delta_{\mathbf{ii}}:=\mathbf{A_{ii}}-\bar{\mathbf{A}}_{\mathbf{ii}}=0$. Then we can write
$\mathbf{b}_{\mathbf{i}}=\mathbf{A_{ii}}\mathbf{x}_{\mathbf{i}}+\delta_{\mathbf{i}} \mathbf{x}$.
Since $\mathbf{b}_{\mathbf{i}}=\mathbf{A_{ii}}\bar{\mathbf{x}}_{\mathbf{i}}$, it follows that $\mathbf{A_{ii}}(\mathbf{x}_\mathbf{i}-\bar{\mathbf{x}}_{\mathbf{i})}+\delta_{\mathbf{i}} \mathbf{x}=0$,
or equivalently $\mathbf{x}_{\mathbf{i}}-\bar{\mathbf{x}}_{\mathbf{i}}=- \mathbf{A_{ii}^{-1}} \delta_{\mathbf{i}} \mathbf{x}$.
Then, we have
\begin{equation}\label{Error-Perturbed-LinSys-Block}
\|\mathbf{x}_{\mathbf{i}} -\bar{\mathbf{x}}_{\mathbf{i}}\|_2 \leq \|\mathbf{A_{ii}^{-1}}\|_2 \|\delta_{\mathbf{i}}\|_2 \|\mathbf{x}\|_2= \lambda_{\mathrm{min}}^{-1}\big(\mathbf{A_{ii}}\big) \|\delta_{\mathbf{i}}\|_2 \|\mathbf{x}\|_2
\end{equation}
Thus, the error estimate (\ref{Error-Estimate-Perturbed-LinSys}) follows from (\ref{Error-Perturbed-LinSys-Block}) and noticing that $\|\mathbf{x}-\bar{\mathbf{x}}\|^2_2=\sum_{i=1}^{\mathcal{N}_\mathbf{b}}\|\mathbf{x}_{\mathbf{i}}-\bar{\mathbf{x}}_{\mathbf{i}}\|^2_2$.
\end{proof}

\begin{theorem}
Let $\hat{\mathcal{S}}$ and $\check{\mathcal{S}}$ be the operators defined by (\ref{S-hat-operator-def}) and (\ref{S-check-def}), respectively. Let $\delta_{H_L}:= \hat{\mathbb{A}}_{H_L}-\check{\mathbb{A}}_{H_L}$, and $\delta_{\mathbf{j}}$ be the rectangular submatrix of $\delta_{H_L}$ obtained by collecting the rows of $\delta_{H_L}$ whose indices coincide with those of the rows of $\check{\mathbb{A}}_{H_L}$ containing its $j$-th block. Then, the following error estimate holds:
\begin{eqnarray}\label{Error-Estimate-4}
\nonumber \|\hat{\mathcal{S}}f-\check{\mathcal{S}}f\|_a &\leq& \left[\Big(\max_{\mathbf{j}\in\{1,\ldots,N_L\}}\lambda_{\mathrm{max}}\big(\hat{\mathbb{A}}_{\mathbf{jj}}^{(H_L)}\big) + \|\delta_{H_L}\|_{2}\Big) \sum_{i=1}^{N_L}\frac{\|\delta_{\mathbf{i}}\|^2_2}{\lambda^{2}_{\mathrm{min}}\big(\hat{\mathbb{A}}_{\mathbf{ii}}^{(H_L)}\big)} \right]^{\frac{1}{2}} \\ 
&\times & \frac{\mathrm{diam}(\Omega)\|f\|_{L^2(\Omega)}}{\pi\sqrt{\lambda_{\mathrm{min}}(\hat{\mathbb{A}}_{H_L})\alpha}},
\end{eqnarray}
where $\hat{\mathbb{A}}_{\mathbf{ii}}^{(H_L)}$ is the $i$-th diagonal block of $\check{\mathbb{A}}_{H_L}$ and $\hat{\mathbb{A}}_{H_L}$.
\end{theorem}

\begin{proof}
Let $f\in L^2(\Omega)$, $\hat{\mathcal{S}}f=\hat{u}=\sum_{i=1}^{N_L}\hat{c}_i \hat{\varphi}_{i}$, $\check{\mathcal{S}}f=\check{u}=\sum_{i=1}^{N_L}\check{c}_i \hat{\varphi}_{i}$ and $\mathbf{c}_d=[\hat{c}_1-\check{c}_1,\ldots,\hat{c}_{N_L}-\check{c}_{N_L}]$. Using \cref{lem:discrading-off-block-entries} and \cref{Riesz-basis}, we have
\begin{eqnarray*}
 \|\hat{u}-\check{u}\|^{2}_{a}&=&\Big\|\sum_{i=1}^{N_L}(\hat{c}_i-\check{c}_i)\hat{\varphi}_{i} \Big\|^2_{a} = \mathbf{c}_{d} \hat{\mathbb{A}}_{H_L}^{(L)} \mathbf{c}_{d}^T =\mathbf{c}_{d} \check{\mathbb{A}}_{H_L} \mathbf{c}_{d}^T +\mathbf{c}_{d} \delta_{H_L} \mathbf{c}_{d}^T\\
&\leq & \left(\|\check{\mathbb{A}}_{H_L}\|_2 + \|\delta_{H_L}\|_{2}\right)\|\mathbf{c}_{d}\|_2^2 \leq \left(\|\check{\mathbb{A}}_{H_L}\|_2 + \|\delta_{H_L}\|_{2}\right) \sum_{i=1}^{L}\frac{\|\delta_{\mathbf{i}}\|^2_2}{\lambda^{2}_{\mathrm{min}}\big(\hat{\mathbb{A}}_{\mathbf{ii}}^{(H_L)}\big)} \|\hat{\mathbf{c}}\|_2^2 \\
 &\leq& \left(\max_{\mathbf{j}\in\{1,\ldots,N_L\}}\lambda_{\mathrm{max}}\big(\hat{\mathbb{A}}_{\mathbf{jj}}^{(H_L)}\big) + \|\delta_{H_L}\|_{2}\right) \sum_{i=1}^{L}\frac{\|\delta_{\mathbf{i}}\|^2_2}{\lambda^{2}_{\mathrm{min}}\big(\hat{\mathbb{A}}_{\mathbf{ii}}^{(H_L)}\big)}\frac{\mathrm{diam}^2(\Omega)\|f\|^2_{L^2(\Omega)}}{\pi^2\lambda_{\mathrm{min}}(\hat{\mathbb{A}}_{H_L})\alpha}\cdot
\end{eqnarray*}
\end{proof}

\begin{remark}
By construction, the localized hierarchical basis functions are normalized with respect to the $\|\cdot\|_{a}$ norm. Hence, $a(\hat{\varphi}_{i},\hat{\varphi}_{j})\leq 1$ for any $i,j\in \{1,\ldots,N_L\}$. Using the Gershgorin Circle Theorem, we observe that $\max_{\mathbf{j}\in\{1,\ldots,L\}}\lambda_{\mathrm{max}}\big(\hat{\mathbb{A}}_{\mathbf{jj}}^{(H_L)}\big)$ is bounded by a small number (pessimistically by the maximum number of basis functions whose supports partially overlap that of a given basis function). This implies that the factor $\big(\max_{\mathbf{j}\in\{1,\ldots,L\}}\lambda_{\mathrm{max}}\big(\hat{\mathbb{A}}_{\mathbf{jj}}^{(H_L)}\big)+  \|\delta_{H_L}\|_{2}\big)^{\frac{1}{2}}$ in (\ref{Error-Estimate-4}) is expected to be small. Consequently, our estimate in (\ref{Error-Estimate-4}) tells us that the truncation error could become relatively more significant as the smallest eigenvalue among the blocks decreases. Additionally, we see from (\ref{Error-Estimate-4}) that the r.h.s.\ increases with the number of levels. However, based on Remark \ref{rmk:a-Inner-Product-value}, the quantities $\|\delta_{\mathbf{i}}\|$ can be made as small as necessary by choosing large enough supporting patches for the basis functions, allowing for the control of the above estimate and a small `matrix-truncation error'.
\end{remark}
%\jc{
%\begin{remark}
%In \cite[Lemma 6]{FeP20}, it is observed that the maximum and minimum eigenvalues of the stiffness matrix blocks for the hierarchical LOD case are independent of the mesh size and only depend on the contrast $\beta/\alpha$. As we will see in the numerical experiments section, this property also holds for the SLOD case, which is an expected corollary from the explanation given in Appendix \ref{Appendix-SLOD-Stability}. 
%\end{remark}
%}

\subsection{Third compression stage: sparsification by approximating inverses of diagonal blocks of $\check{\mathbb{A}}_{H_L}$}
The inverse of the block-diagonal matrix $\check{\mathbb{A}}_{H_L}$ is another block-diagonal matrix whose diagonal blocks are the inverses of the corresponding blocks of $\check{\mathbb{A}}_{H_L}$. Note that the inverse of each block is a full matrix (see relationship of the inverse of a matrix and its characteristic polynomial using the Cayley-Hamilton theorem). 

The $i$-th column of the inverse of the $j$-th diagonal block of $\check{\mathbb{A}}_{H_L}$ can be obtained by solving the linear system of equations
\begin{equation*}
\hat{\mathbb{A}}_{\mathbf{jj}}^{(H_L)}\mathbf{c^{(i)}}=\mathbf{e_i},
\end{equation*}
where $\hat{\mathbb{A}}_{\mathbf{jj}}^{(H_L)}$ denotes the $j$-th diagonal block of both $\check{\mathbb{A}}_{H_L}$ and $\hat{\mathbb{A}}_{H_L}$.
If we solve the above linear system iteratively with the Conjugate Gradient method (CG), we obtain the following error bound.
\begin{equation*}
\big\|\mathbf{c^{(i)}}-\mathbf{c}^{(\mathbf{i},k)}\big\|_{\hat{\mathbb{A}}_{\mathbf{jj}}^{(H_L)}}\leq 2\left(\frac{\sqrt{\kappa\big(\hat{\mathbb{A}}_{\mathbf{jj}}^{(H_L)}\big)}-1}{\sqrt{\kappa\big(\hat{\mathbb{A}}_{\mathbf{jj}}^{(H_L)}\big)}+1}\right)^{k}\big\|\mathbf{c^{(i)}}-\mathbf{c}^{(\mathbf{i},0)}\big\|_{\hat{\mathbb{A}}_{\mathbf{jj}}^{(H_L)}}
\end{equation*}
where the $k$-th CG iterate is denoted by $\mathbf{c}^{(\mathbf{i},k)}$. Thus, the column $\mathbf{c^{(i)}}$ could be well approximated with a few CG iterations provided the condition number of $\hat{\mathbb{A}}_{\mathbf{jj}}^{(H_L)}$ is small. 

For an idea of the sparsity degree of the approximate inverses of the blocks $\check{\mathbb{A}}_{H_L}$, consider a uniform rectangular Cartesian mesh and a uniform patch-order parameter $m_j$ at level $j$. Then, if we start the CG iterations with $\mathbf{c}^{(\mathbf{i},0)}=\mathbf{0}$, the number of non-zeros of the $k$-th iterate can be bounded by
\begin{equation}\label{nnz-bound}
\textbf{nnz}\big(\mathbf{c}^{(\mathbf{i},k)}\big)\leq 8(2^d-1)\left(2(2k-1)m_{j}^2+m_{j}\right),
\end{equation}
where the bound is tight (i.e., the bound is attained for some $\mathbf{i}\leq \# \hat{\mathcal{B}}_{j}=(2^d-1)N_{j-1} $) provided $k$ is such that the r.h.s.\ of the above inequality is less than $\# \hat{\mathcal{B}}_{j}$.
\begin{remark}
The bound in (\ref{nnz-bound}) can be obtained by considering $\hat{\mathbb{A}}_{\mathbf{jj}}^{(H_L)}$ as the adjacency matrix of a weighted graph $\mathcal{G}$. Then, the number of non-zeros of $\mathbf{c}^{(\mathbf{i},k)}$ (for $k\geq 1$) equals the number of vertices of $\mathcal{G}$ connected to the vertex $\mathbf{i}$ by a path (union of edges) of length not greater than $k$, considering the vertex $\mathbf{i}$ connected to itself.
\end{remark}

From what is stated above, we can see that if the condition number of $\hat{\mathbb{A}}_{\mathbf{jj}}^{(H_L)}$ is small, we can accurately approximate its inverse by a cheaply-computable sparse matrix $\mathbb{S}^{(k)}\in \mathbb{R}^{N_{L}\times N_{L}}$, where
\begin{equation*}
\mathbb{S}^{(k)}_{i,j}=\mathbf{c}^{(\mathbf{j},k)}_i.
\end{equation*}

Define $\bar{\mathfrak{S}}:\mathbb{R}^{N_L} \rightarrow \mathbb{R}^{N_L}$ such that $\bar{\mathfrak{S}}(\mathbf{x})=\mathbb{S}^{(k)}\mathbf{x}$. Let $\mathcal{S}:L^2(\Omega)\rightarrow \hat{\mathcal{V}}_{H_{L}}$ be given by
\begin{equation}\label{S-third-compression-operator}
\bar{\mathcal{S}}=\mathcal{L}\circ \bar{\mathfrak{S}} \circ \mathcal{R}.
\end{equation}
The next theorem estimates of how good $\bar{\mathcal{S}}$ is as an approximation of the operator $\check{\mathcal{S}}$.

\begin{theorem}\label{Thm:Third-compression}
Let $\check{\mathcal{S}}$ and $\bar{\mathcal{S}}$ be the operators defined by (\ref{S-check-def}) and (\ref{S-third-compression-operator}), respectively. Let $\delta_{H_L}:= \hat{\mathbb{A}}_{H_L}-\check{\mathbb{A}}_{H_L}$. Moreover, let $\delta^{\text{CG}}>0$ and $k_{\delta^{\text{CG}}}\in \mathbb{N}$ be such that $\left\|\check{\mathbb{A}}_{H_L}^{-1}-\mathbb{S}^{(k)}\right\|_2\leq \delta^{\text{CG}}$. Then, the following estimate holds.
\begin{equation}\label{Third-Compression-Estimate-Statement}
\|\check{\mathcal{S}}f-\bar{\mathcal{S}}f\|_{a}\leq \Big(\max_{\mathbf{j}\in\{1,\ldots,N_L\}}\lambda_{\mathrm{max}}\big(\hat{\mathbb{A}}_{\mathbf{jj}}^{(H_L)}\big) + \left\|\delta_{H_L} \right\|_2 \Big)^{\frac{1}{2}}\delta^{\text{CG}}\sqrt{N_E L}\frac{\mathrm{diam}(\Omega)}{\pi \sqrt{\alpha}}\|f\|_{L^2(\Omega)},
\end{equation}
where $\hat{\mathbb{A}}_{\mathbf{jj}}^{(H_L)}$ is the $j$-th diagonal block of $\check{\mathbb{A}}_{H_L}$. 
\end{theorem}

\begin{proof}
Let $f\in L^2(\Omega)$, $\check{\mathcal{S}}f=\check{u}=\sum_{i=1}^{N_L}\check{d}_i \hat{\varphi}_{i}$, $\bar{\mathcal{S}}f=\bar{u}=\sum_{i=1}^{N_L}\bar{d}_i \hat{\varphi}_{i}$, $\mathbf{d}=[\check{d}_1-\bar{d}_1,\ldots,\check{d}_{N_L}-\bar{d}_{N_L}]$, and $\Omega_i=\mathrm{supp}(\hat{\varphi}_i)$. Note that
\begin{eqnarray}\label{Rf-bound}
\nonumber \|\mathcal{R}f\|_2^2 = \sum_{i=1}^{N_L}(f,\hat{\varphi}_i)_{L^2(\Omega)}^2\leq  \sum_{i=1}^{N_L}\|f\|^2_{L^2(\Omega_i)} \|\hat{\varphi}_i\|^2_{L^2(\Omega_i)} 
 &\leq&  \sum_{i=1}^{N_L}\|f\|^2_{L^2(\Omega_i)}\left(\frac{\mathrm{diam}(\Omega)}{\pi \sqrt{\alpha}}\|\hat{\varphi}_i\|_{a}\right)^2\\ 
&\leq&  \frac{\mathrm{diam}^2(\Omega)}{\pi^2\alpha} N_E L \|f\|^2_{L^2(\Omega)} ,
\end{eqnarray}
since $\|\hat{\varphi}_i\|_a=1$ by construction. Then, noticing that $\mathbf{d}=\big( \check{\mathbb A}_{H_L}^{-1}-\mathbb{S}^{(k)}\big)\mathcal{R}f$, and using \eqref{Rf-bound}, it follows that
\begin{eqnarray}\label{Third-Compression-Estimate-proof}
\nonumber \| \check{u}- \bar{u} \|^{2}_{a}&=&\Big\|\sum_{i=1}^{N_L}(\check{d}_i-\bar{d}_i)  \hat{\varphi}_{i} \Big\|_{a}^{2}=\mathbf{d}^{T}\hat{\mathbb{A}}_{H_L}\mathbf{d}=\mathbf{d}^{T}\left(\check{\mathbb{A}}_{H_L}+\delta_{H_L}\right)\mathbf{d}\leq  \left( \left\|\check{\mathbb{A}}_{H_L}\right\|_2 + \left\|\delta \right\|_2 \right)\|\mathbf{d}\|_2^2 \\
\nonumber &\leq& \left( \left\|\check{\mathbb{A}}_{H_L}\right\|_2 + \left\|\delta \right\|_2 \right)\left(\|\check{\mathbb A}_{H_L}^{-1}-\mathbb{S}^{(k)}\|_2 \|\mathcal{R}f\|_2\right)^2\\ 
& \leq &\left( \left\|\check{\mathbb{A}}_{H_L}\right\|_2 + \left\|\delta \right\|_2 \right)\Big(\delta^{\text{CG}}\frac{\mathrm{diam}(\Omega)}{\pi \sqrt{\alpha}}\sqrt{N_E L}\|f\|_{L^2(\Omega)}\Big)^2.
\end{eqnarray}

From (\ref{Third-Compression-Estimate-proof}), and noticing that 
$\left\|\check{\mathbb{A}}_{H_L}\right\|_2=\max_{\mathbf{j}\in\{1,\ldots,L\}}\lambda_{\mathrm{max}}\big(\hat{\mathbb{A}}_{\mathbf{jj}}^{(H_L)}\big)$,
the estimate in (\ref{Third-Compression-Estimate-Statement}) follows.
\end{proof}

\subsection{Fourth compression stage: sparsification by discarding entries of $\mathbb{S}^{(k)}$ with absolute values below a prescribed tolerance}
Let $\mathbb{S}_{\epsilon}\in \mathbb{R}^{N_L \times N_L}$ be the matrix obtained after discarding the entries of $\mathbb{S}^{(k)}$ smaller than a given tolerance $\epsilon>0$, and let $N_{\epsilon}$ be the maximum possible number of non-zero entries in any row or column of $\mathbb{S}_{\epsilon}$. It follows that
\begin{equation}\label{matrix-difference-2-norm-bound-by-max}
\|\mathbb{S}_{\epsilon}-\mathbb{S}^{(k)}\|_2\leq\sqrt{\|\mathbb{S}_{\epsilon}-\mathbb{S}^{(k)}\|_1 \|\mathbb{S}_{\epsilon}-\mathbb{S}^{(k)}\|_{\infty}}\leq N_{\epsilon} \epsilon.
\end{equation}
With $\mathfrak{S}_{\epsilon}:\mathbb{R}^{N_L} \rightarrow \mathbb{R}^{N_L}$ such that $\mathfrak{S}_{\epsilon}(\mathbf{x})=\mathbb{S}_{\epsilon}\mathbf{x}$, define $\mathcal{S}_\epsilon:L^{2}(\Omega)\rightarrow \hat{\mathcal{V}}_{H_{L}}$ such that
\begin{equation}\label{Fourth-compression-operator}
\mathcal{S}_{\epsilon}=\mathcal{L}\circ \mathfrak{S}_{\epsilon} \circ \mathcal{R}.
\end{equation}

With the same procedure employed in Theorem \ref{Thm:Third-compression}, and using (\ref{matrix-difference-2-norm-bound-by-max}), we obtain the following approximation result.
\begin{theorem}
Let $\epsilon>0$, and $\mathcal{S}_{\epsilon}$ and $\bar{\mathcal{S}}$ be the operators defined in (\ref{S-third-compression-operator}) and (\ref{Fourth-compression-operator}), respectively. Then, the estimate
\begin{equation*}
\|\bar{\mathcal{S}}f-\mathcal{S}_{\epsilon}f\|_{a}\leq \left(\max_{\mathbf{j}\in\{1,\ldots,N_L\}}\lambda_{\mathrm{max}}\big(\hat{\mathbb{A}}_{\mathbf{jj}}^{(H_L)}\big) + \left\|\delta_{H_L} \right\|_2 \right)^{\frac{1}{2}}\epsilon \sqrt{N_E L}N_{\epsilon}\frac{\mathrm{diam}(\Omega)}{\pi \sqrt{\alpha}}\|f\|_{L^2(\Omega)}
\end{equation*}
holds, where $\hat{\mathbb{A}}_{\mathbf{jj}}^{(H_L)}$ is the $j$-th diagonal block of $\check{\mathbb{A}}_{H_L}$.
\end{theorem}

\subsection{Overall compression error of the approximate solution}
Collecting the approximation results from the four compressions stages, we obtain the following error of the approximate solution obtained with the sparse-compressed solution operator $\mathcal{S}=\mathcal{S}_{\epsilon}$.

\begin{theorem}
The error in the energy norm of the approximate solution to \eqref{Continuous-Formulation} obtained with the sparse-compressed operator $\mathcal{S}=\mathcal{S}_{\epsilon}$ arising after the four compression stages is given by
\begin{eqnarray}\label{Overall-Error-Estimate}
\nonumber \|\mathcal{A}^{-1}f-\mathcal{S}f\|_a &\leq& \|\mathcal{A}^{-1}f-\mathcal{A}_{h}^{-1}f\|_a \\ \nonumber 
&+& \frac{H_L}{\pi \sqrt{\alpha}} \left\|f-\Pi_{H_L}f\right\|_{L^2(\Omega)}\\ \nonumber
&+& C_{\sharp}\frac{\sqrt{N_E L}}{\sqrt{\alpha}} \left(  \frac{\zeta \sigma }{\sqrt{\lambda_{\mathrm{min}}(\mathbf{G})L}}+\frac{\delta^{\text{tr}}}{\sqrt{\lambda_{\mathrm{min}}\big(\hat{\mathbb{A}}_{H_L}\big)}}+\delta^{\text{CG}}+N_{\epsilon}\epsilon \right) \left\|f\right\|_{L^2(\Omega)},
\end{eqnarray}
where $C_{\sharp}$ is given by
\begin{equation*}
C_{\sharp}= \left(\max_{\mathbf{j}\in\{1,\ldots,N_L\}}\lambda_{\mathrm{max}}\left(\hat{\mathbb{A}}_{\mathbf{jj}}^{(H_L)}\right) + \left\|\delta_{H_L} \right\|_2 \right)^{\frac{1}{2}}\left(1+\frac{\mathrm{diam}(\Omega)}{\pi}\right),
%\leq \sqrt{N_L}(1+\sigma)\left(1+\frac{\mathrm{diam}(\Omega)}{\pi}\right),
\end{equation*}
$\delta^{CG}>0$ and $\epsilon>0$ are small parameters proportional to the prescribed accuracy, and
\begin{equation*}
\delta^{\text{tr}}=\frac{1}{\sqrt{N_E}}\max_{\mathbf{i}\in\{1,\ldots,L\}}\frac{\|\delta_{\mathbf{i}}\|_2}{\lambda_{\mathrm{min}}\big(\hat{\mathbb{A}}_{\mathbf{ii}}^{(H_L)}\big)}
\end{equation*}
is a quantity arising from the truncation of off-block-diagonal entries of $\hat{\mathbb{A}}_{H_L}$.
\end{theorem}

\begin{remark}\label{rem:blockwise solution}
Alternatively, in particular when the condition number of the diagonal blocks of $\hat{\mathbb{A}}_{H_L}$ are not sufficiently small to justify the computation of the sparse-compressed operator $\mathcal{S}$, we could obtain an approximate solution to \eqref{Continuous-Formulation} by finding $u_{h,L}^{m}=\sum_{\ell=0}^{L}u^{m}_{\ell}\in \hat{\mathcal{V}}_{L}$, where $u^{m}_{\ell}\in \mathbf{span}\;\hat{\mathcal{B}}_{\ell}$ is obtained by solving 
\begin{equation*}
a(u^{m}_{\ell},v)=(f,v)_{L^2(\Omega)}\quad \text{for all }v \in \mathbf{span}\;\hat{\mathcal{B}}_{\ell}.
\end{equation*}
Note that in this case $u_{h,L}^{m}$ is equivalent to the approximate solution obtained after the first two compression stages previously discussed.
\end{remark}

\section{Numerical Experiments}\label{sec:Numerical-Experiments}
In this section, we present numerical experiments that validate our preceding theoretical findings. We consider two types of coefficients: a highly-heterogeneous piecewise-constant one, and coefficients with high-contrast channels; see \cref{fig:coeffs} for an example of each. In all cases, we utilize uniform Cartesian meshes over the unit square $\Omega = (0,1)^2$ where the mesh size denotes the side length of the elements. 

For the solution of local patch problems we use a fine mesh (see definition of the local solution operator $\mathcal{A}^{-1}_{h,\omega}$) with mesh size $h=2^{-9}$ for the piecewise-constant coefficient and with mesh size $h=2^{-10}$ for the high-contrast channel case. We denote the fully discrete numerical approximation of the solution to \eqref{Continuous-Formulation} by $u_{h,L}^m$. To calculate this approximation, we do not explicitly compute the inverse of the stiffness matrix or its blocks. Instead, $u_{h,L}^m$ is calculated according to \Cref{rem:blockwise solution}. For stabilization we use condition \eqref{SLOD-stability-condition} with $\delta_s = 0.5$. Additionally, we compute a reference solution $u_h$ using the standard finite element method with the same fine mesh size $h$ and $\mathcal Q_1$-finite elements.

\begin{figure}
	\centering
	\includegraphics[width=1\linewidth]{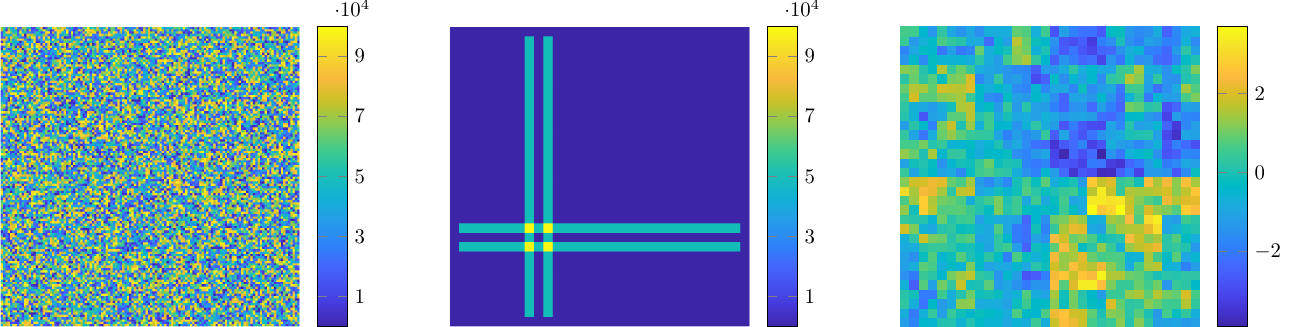}
	\caption{Piecewise-constant coefficient (left), coefficient with high-contrast channels (middle) and piecewise-constant right-hand side $f\in L^2(\Omega)$ with respect to the mesh with mesh size $2^{-5}$ (right).}
	\label{fig:coeffs}
\end{figure}

In our numerical experiments, we consider two types of right-hand sides. First, a smooth right-hand side $f\in H^1(\Omega)$ given by
\begin{equation}\label{smooth_f}
	f(x_1,x_2) = 2\pi^2\sin(\pi x_1)\sin(\pi x_2).
\end{equation}
Second, we employ a piecewise-constant right-hand side $f\in L^2(\Omega)$ with respect to the mesh $\mathcal T_f$ with mesh size $2^{-5}$; see \cref{fig:coeffs}. More precisely, we choose $f=\sum_{\ell=0}^{5}f_{\ell}$ with $f_{\ell}\in \mathbb{Q}^0(\T_{{\ell}})$ and $H_{\ell}=2^{-\ell}$, where the values of each $f_{\ell}$ are randomly chosen within the interval $[-1,1]$. With this choice, we ensure that the contribution of each level to the approximate solution is non-zero.

\subsection{Piecewise-constant coefficients}\label{Subsec:NumExpPwct}
In our first set of experiments, we select a high-contrast coefficient $\mathbf{A}$ which is piecewise-constant with respect to the mesh of mesh size $2^{-7}$. The coefficient assumes independent and identically distributed element values ranging between $\alpha=1$ and $\beta=10^5$. However, in the specific realization we consider, these boundary values are not assumed, resulting in an actual contrast of $10^4$.

\cref{fig:full_stiffness_pcA} illustrates the complete stiffness matrix associated with the HSLOD basis for different values of the patch order $m$. As discussed in \cref{rmk:a-Inner-Product-value}, the practical hierarchical method does not yield a fully $a$-orthogonal basis, leading to the appearance of some non-zero off-block-diagonal entries in the stiffness matrix. However, the $a$-inner product between two non-orthogonal basis functions at distinct levels is small and diminishes further with increasing $m$. Therefore, the block-diagonal stiffness matrix, obtained by discarding all off-block-diagonal entries, and its corresponding inverse serve as good approximations to the full stiffness matrix and its inverse, respectively, particularly for larger $m$.

Small condition numbers of the individual blocks are crucial for the fast computation of their approximate inverses using the Conjugate Gradient method (CG). The condition numbers of the blocks along with the relative residual after seven CG iterations are presented in \cref{table:cg_properties_constA}. The blocks of the stiffness matrix exhibit good conditioning with condition numbers that appear stable across finer levels, leading to a good approximation of the inverse with only a few CG iterations.

The inverse of the block-diagonal stiffness matrix (2nd compression stage), its CG approximation (3rd compression stage) and the sparsification of this approximation (4th compression stage) are illustrated in \cref{fig:sp_inverse_pcA}. After the fourth compression stage, the number of non-zero entries is reduced by a factor of ten compared to the complete block-diagonal inverse, and by a factor greater than two compared to the number of non-zeros of the stiffness matrix used to obtain $u_h$. In \Cref{table:cg_properties_constA}, we observe the relative energy error of the approximate solution obtained after the fourth compression stages using the smooth right-hand side.

\begin{figure}
	\includegraphics[width=1\linewidth]{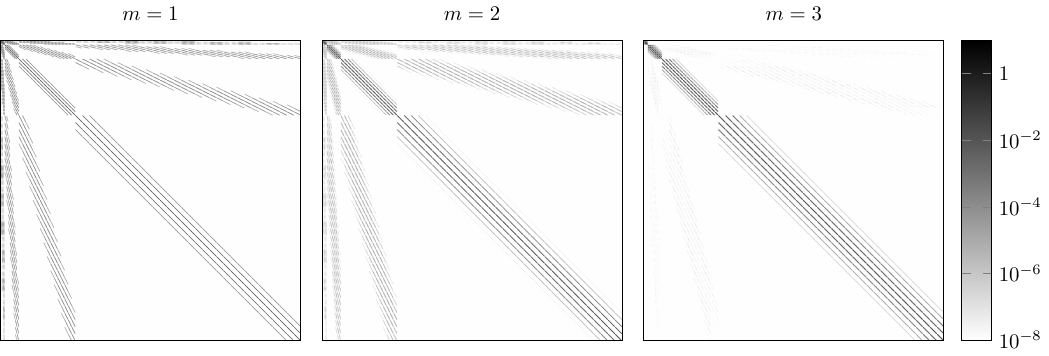}\hfill
	\caption{Sparsity pattern of the complete stiffness matrix for varying patch orders $m$. Utilizing a logarithmic gray-scale for color representation, darker shades indicate larger magnitudes of the corresponding entries, while zero entries are depicted in white.}
	\label{fig:full_stiffness_pcA}
\end{figure}

\begin{table}%[htbp]
%\footnotesize
	\caption{Properties of the diagonal blocks of the HSLOD stiffness matrix, and  relative energy error of the sparse-compressed operator, for a piecewise-constant coefficient and $m=2$. The condition number of the blocks is displayed for levels $\ell = 1,\dots,6$. Here, $S_{\ell}$ denotes the sparse-compressed approximate-solution operator resulting after the fourth compression stages, with $\ell$ levels of the hierarchy, seven maximum CG iterations to approximate the block inverses, and entries cut-off tolerance $\epsilon = 10^{-5}$.}\label{table:cg_properties_constA}
\begin{center}
  \begin{tblr}{ |p{3cm}||p{1.3cm}|p{1.3cm}|p{1.2cm}|p{1.2cm}|p{1.2cm}|p{1.2cm}|}
  	\hline
  	$\ell$   &1&2& 3&4&5&6\\ [-0.5ex]
  	\hline
  	$H_\ell$   & $2^{-1}$& $2^{-2}$&$2^{-3}$&$2^{-4}$&$2^{-5}$&$2^{-6}$\\[-0.5ex]
  	\hline
  	$\text{cond}(\hat{\mathbb{A}}_{\ell\ell}^{H_\ell})$  & 2&  6&21&14&23&22\\[-0.4ex]
%  	\hline
%  	$\|\mathbf{I}-\hat{\mathbb{A}}_{\ell\ell}^{H_\ell}\cdot\mathbb{S}^{(7)}_{\ell \ell}\|_2$  & $<10^{-15}$& $2.8\cdot10^{-5}$ &  0.015 & 0.0046 & 0.015 & 0.026 \\[-0.4ex]
  	\hline
  $\|u_h-\mathcal{S}_{\ell}f \|_a$/$\|u_h\|_a$ &	$0.2298$ &   $0.0493$ &   $0.0118$  &  $0.0030$  &  $0.0008$  &  $0.0003$\\
  \hline 	
  \end{tblr}
\end{center}
\end{table}

\begin{figure}
\includegraphics[width=1\linewidth]{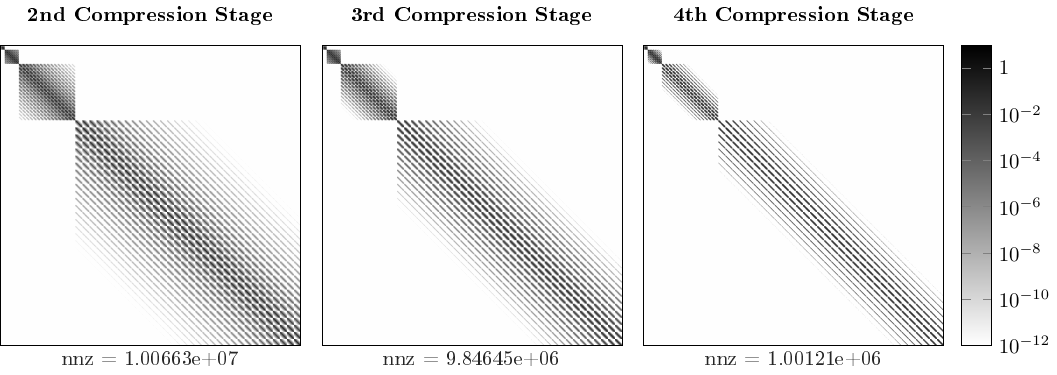}\hfill
\caption{Sparsity pattern and number of non-zeros of the inverse of the stiffness matrix after the different compression stages. Left: inverse of block-diagonal matrix; middle: CG approximation of each inverted block after seven iterations; right: discarding entries of the CG approximation with absolute value smaller than $10^{-5}$. The patch order is set to $m=2$. }
\label{fig:sp_inverse_pcA}
\end{figure}

\cref{fig:H1error_d2_pcA_smoothf} illustrates the relative energy errors of the HSLOD method using the first two compression stages, for different values of the patch order $m$ and coarse mesh size $H_L$, and for both types of right-hand sides $f$. Additionally, the relative energy errors of the hierarchical LOD (HLOD) method are included for comparison. Since the HSLOD basis functions are derived from the corrections of LOD functions, the HLOD method can be obtained by simply omitting these corrections. The HLOD method achieves better accuracy then the gamblets-based method of \cite{Owh17} and exhibits similar error behavior as the stabilized version of \cite{FeP20} introduced in \cite{HaPe21}. Notably, the HSLOD consistently outperforms the HLOD across all displayed parameters $m$, demonstrating superior accuracy. 

In the case of a smooth right-hand side, the error in the energy norm obtained with the HSLOD method exhibits the optimal convergence rate of $\mathcal O(H^2)$ for all levels when $m>1$. For a piecewise-constant right-hand side $f\in L^2(\Omega)$, the expected error rate of $\mathcal O(H)$ is observed, provided that no mesh in the hierarchy resolves the underlying mesh of $f$. However, if the finest mesh in the hierarchy does resolve $\T_f$, the HSLOD solution is exact up to localization error and the error due to omitting the off-block-diagonal entries in the stiffness matrix. This error scales like $\sigma$ (cf. \Cref{rmk:a-Inner-Product-value}), resulting in a superexponential decay over $m$ and significantly smaller errors compared to the HLOD. 

It is also noteworthy that variations in the parameter $\beta$ do not produce significant differences in the error plots or condition numbers. Therefore, for this type of coefficient, the HSLOD method performs well even in the presence of high contrasts.

\begin{figure}
	\centering
	\includegraphics[width=.48\linewidth]{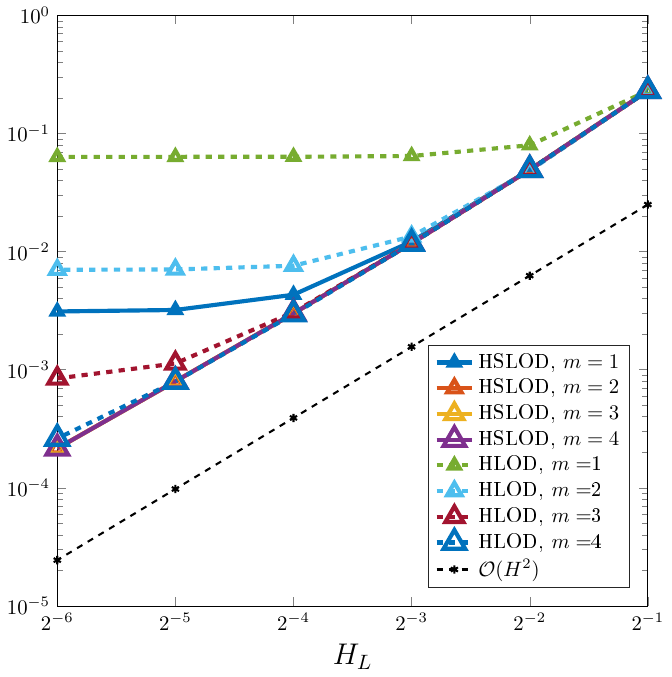}\hfill
	\includegraphics[width=.4855\linewidth]{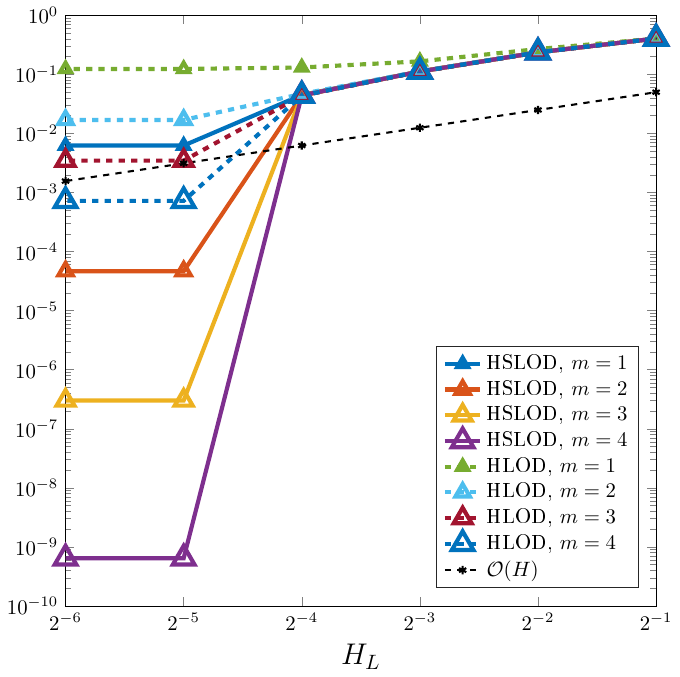}
	\caption{Plot of the relative energy errors $\|u_h - u_{h,L}^m \|_a/\|u_h\|_a$ of the HSLOD and the HLOD in dependence of the mesh size $H_L$ for a piecewise-constant coefficient $\Ab$. Left: errors for a smooth right-hand side, right: errors for a piecewise-constant right-hand side.}
	\label{fig:H1error_d2_pcA_smoothf}
\end{figure}

\subsection{High-contrast channels}
We repeat the above experiments with the second coefficient from \cref{fig:coeffs}, which exhibits high-contrast channels, leading to increasing difficulty in the problem. The channels are defined on a mesh with mesh size $2^{-5}$. We denote the level on which the channels are defined as $\ell_{\Ab}$. The precise definition of the coefficient is as follows:
\begin{equation*}
	\Ab(x)=\Ab(x_1,x_2) \coloneqq \Ab_1(x_1,x_2) + \Ab_1(x_2,x_1),
\end{equation*}
with
\begin{equation*}
	\Ab_1(x)\coloneqq 
	\begin{cases}
		\beta/2, & x\in [\frac{8}{32},\frac{9}{32}] \times [\frac{1}{32},\frac{31}{32}] \cup [\frac{10}{32},\frac{11}{32}] \times [\frac{1}{32},\frac{31}{32}] \\
		1/2, &\text{elsewhere}.
	\end{cases}
\end{equation*}

For this setup, with the smooth right-hand side given by \eqref{smooth_f}, \cref{fig:H1error_d2_hcA_smoothf} illustrates the relative energy errors of the HSLOD and the HLOD methods for different combinations of contrasts and patch orders. Note that no patch contains a whole channel on levels $\ell\geq\ell_{\Ab}$, i.e., on levels associated with finer meshes than the underlying mesh of $\Ab$. We observed that for higher contrasts, larger values of $m$ are necessary to achieve the optimal convergence rate of $\mathcal O(H^2)$ over all levels. As in the case of the highly-heterogeneous piecewise-constant coefficient, the HSLOD consistently outperforms the HLOD for all cases shown, leading to improvements over \cite{PeS16} in particular. Especially for higher contrasts, the superiority of the HSLOD is evident. 

\begin{figure}
	\centering
	\includegraphics[width=.7\linewidth]{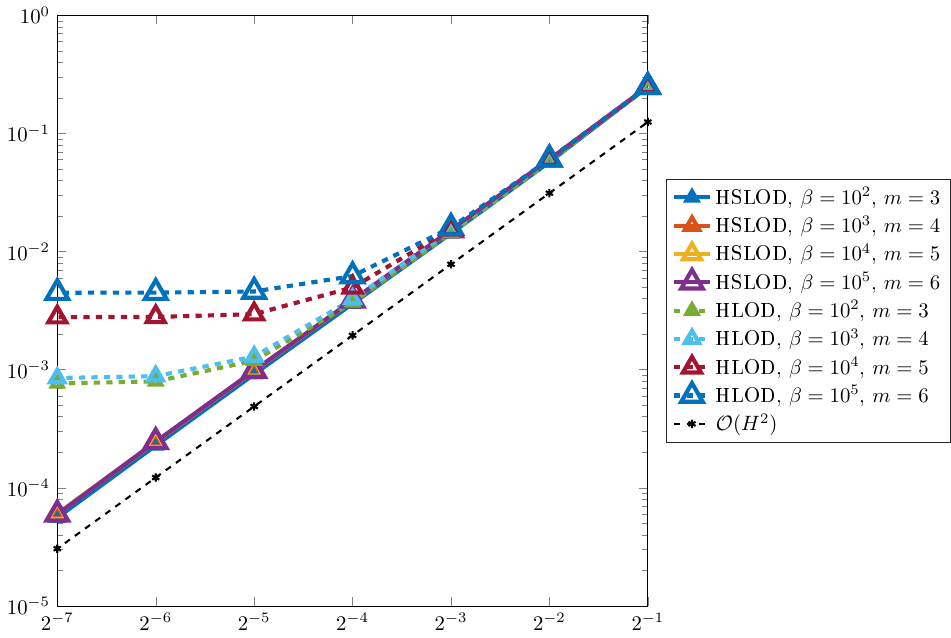}
	\caption{Plot of the relative energy errors $\|u_h - u_{h,L}^m \|_a/\|u_h\|_a$  of the HSLOD and the HLOD in dependence of the mesh size $H_L$ for the coefficient with high-contrast channels and a smooth right-hand-side. Various combinations of $m$ and $\beta$ are displayed.}
	\label{fig:H1error_d2_hcA_smoothf}
\end{figure}

The condition numbers of the blocks of the HSLOD stiffness matrix for the same combinations of $m$ and $\beta$ as in \cref{fig:H1error_d2_hcA_smoothf} are shown in \cref{table:cg_properties_hcA}. For these cases, the condition numbers are high and grow with increasing $\beta$ for levels with mesh sizes close to the width of the channel. For the other levels, however, the condition numbers remain stable over different contrasts. Note that this stability is only observed when the patch order is increased for higher contrasts, leading to our choice of $m = k+1$ for $\beta = 10^k$.

\begin{table}%[htbp]
	\caption{Condition numbers of the blocks of the stiffness matrix for the high-contrast channel coefficient for varying contrasts and patch orders, rounded to two significant figures.}\label{table:cg_properties_hcA}
\begin{center}
  \begin{tblr}{ |p{2.4cm}||p{0.9cm}|p{0.9cm}|p{0.9cm}|p{1.5cm}|p{1cm}|p{0.9cm}|p{0.9cm}|}\hline
		$H_\ell$   & $2^{-1}$& $2^{-2}$&$2^{-3}$&$2^{-4}$&$2^{-5}$&$2^{-6}$& $2^{-7}$\\[-0.5ex]
		\hline
		$\beta = 10^2, m=3$  & 3& 8&  24&110&51&17&17\\[-0.5ex]
		\hline
		$\beta = 10^3, m=4$  & 3& 8&  370&1100&680&76&56\\[-0.5ex]
		\hline
		$\beta = 10^4, m=5$  & 4& 7&  120&21.000&3.000&26&26\\[-0.5ex]
		\hline
		$\beta = 10^5, m=6$  & 4& 7&  120&$1.7\cdot10^6$&33.000&21&21\\[-0.5ex]
		\hline
\end{tblr}
\end{center}
\end{table}

\section{Conclusion}
We derived a superlocalized hierarchical basis for an approximate solution space of (\ref{Continuous-Formulation}), with almost-full $a$-orthogonal basis functions across levels, and where the localization is achieved via the SLOD method. This basis allows for the construction of a compressed operator that serves as a good approximation for the solution operator. The structure of this operator is such that it can be decomposed into the sum of independent operators, allowing for scales decoupling and simultaneous computations that can be exploited to reduce the solution computational time. The resulting method is particularly useful in the presence of rough coefficients, and it can even be extended to the solution of elliptic optimal control problems with rough coefficients \cite{Brenner2022,brenner2024multiscale}.

\appendix
\section{SLOD Stability Analysis}\label{Appendix-SLOD-Stability}
From the definition of SLOD basis functions at level $\ell$ given in Subsection \ref{Stable-SLOD-subsection}, we have for $T \in\mathcal{T}_{\ell,\widetilde{\omega}}$ that
\begin{equation}\label{SLOD_basis_def_apdx}
\hat{\theta}^{\text{SLOD}}_{T,\ell}=\frac{\mathcal{E}_{\widetilde{\omega}}\left(\bar{\psi}^{\text{LOD}}_{T,\ell}+\sum_{K\in \mathcal{T}_{\ell,\widetilde{\omega}}\setminus\{T\}}c_K \bar{\psi}_{K,\ell}^{(T)}\right)}{\left\| \bar{\psi}^{\text{LOD}}_{T,\ell}+\sum_{K\in \mathcal{T}_{\ell,\widetilde{\omega}}\setminus\{T\}}c_K \bar{\psi}_{K,\ell}^{(T)}\right\|_{a_{\widetilde{\omega}}}}\cdot
\end{equation}
It follows that
\begin{equation*}
\Pi_{\ell}\hat{\theta}^{\text{SLOD}}_{T,\ell}=\frac{\chi_{T}+\sum_{K\in \mathcal{T}_{\ell,\widetilde{\omega}}\setminus\{T\}}c_K \chi_{K}}{\left\| \bar{\psi}^{\text{LOD}}_{T,\ell}+\sum_{K\in \mathcal{T}_{\ell,\widetilde{\omega}}\setminus\{T\}}c_K \bar{\psi}_{K,\ell}^{(T)}\right\|_{a_{\widetilde{\omega}}}}\cdot
\end{equation*}
Then, with $z_{T}$ as defined in \eqref{SLOD-stability-condition-weight}, $\|z_{T}^{-1}\Pi_{{\ell}}\hat{\varphi}^{\text{SLOD}}_{T,\ell}-\chi_{T}\|_{L^{\infty}(\Omega)}\rightarrow 0$ implies that $(c_K)_{K\in \mathcal{T}_{\ell,\widetilde{\omega}}\setminus\{T\}}\rightarrow \mathbf{0}$, which, in turn, implies $\|\hat{\theta}^{\text{SLOD}}_{T,\ell}-\mathcal{E}_{\widetilde{\omega}}\left(\bar{\psi}^{\text{LOD}}_{T,\ell}\right) / \|\bar{\psi}^{\text{LOD}}_{T,\ell} \|_{a_{\widetilde{\omega}}} \|_{a}\rightarrow 0$. It is known that the LOD basis is stable and its associated stiffness matrix $\mathbb{A}^{\mathrm{LOD},\ell}$ is well-conditioned \cite{AHP21,FeP20}. Then, since $\|z_{T}^{-1}\Pi_{{\ell}}\varphi^{\text{SLOD}}_{T,\ell}-\chi_{T}\|_{L^{\infty}(\Omega)}\rightarrow 0$ implies $\|\mathbb{A}^{\mathrm{SLOD},\ell}-\mathbb{A}^{\mathrm{LOD},\ell}\|_{\infty}\rightarrow 0$, it follows that
\begin{equation*}
\|z_{T}^{-1}\Pi_{{\ell}}\hat{\varphi}^{\text{SLOD}}_{T,\ell}-\chi_{T}\|_{L^{\infty}(\Omega)}\rightarrow 0 \implies \left| \kappa \left( \mathbb{A}^{\mathrm{SLOD},\ell}  \right) -\kappa\left( \mathbb{A}^{\mathrm{LOD},\ell} \right)\right|\rightarrow 0,
\end{equation*}
i.e., $\mathbb{A}^{\mathrm{SLOD},\ell} $ should be well-conditioned if $\|z_{T}^{-1}\Pi_{{\ell}}\varphi^{\text{SLOD}}_{T,\ell}-\chi_{T}\|_{L^{\infty}(\Omega)}\leq \delta_s$ for small enough $\delta_s\geq 0$.

Note that $\|z_{T}^{-1}\Pi_{{\ell}}\hat{\varphi}^{\text{SLOD}}_{T,\ell}-\chi_{T}\|_{L^{\infty}(\Omega)}$ decreases when the smallest singular value in (\ref{stable_c_computation}) is discarded. Hence, we should obtain a stable basis after discarding enough singular values in (\ref{stable_c_computation}). The superlocalization will be preserved in the resulting basis provided the number of singular values discarded to achieve basis stability is not very large. Note also that the LOD basis is obtained in the limiting case of discarding all singular values in (\ref{stable_c_computation}). The results in \Cref{table:Stability-Condition-Test1} corroborate these analytical findings, where we used $\mu_{\ell}=\max_{T\in \mathcal{T}_{\ell}}\|z_{T}^{-1}\Pi_{{\ell}}\hat{\varphi}^{\text{SLOD}}_{T,\ell}-\chi_{T}\|_{L^{\infty}(\Omega)}$. For these results we employed a highly-heterogeneous piecewise-constant diffusion coefficient with $\beta=2$ and $\alpha=1$, $h=2^{-7}$, $H_0=1$, $\Omega=(0,1)\times(0,1)$, and patch order $m=2$. For the stabilized case, the number of singular values kept in (\ref{stable_c_computation}) is such that the condition \eqref{SLOD-stability-condition} holds with $\delta_s=0.5$ and the retained singular values satisfy $\sigma_1/\sigma_i<10^{15}$. For the unstabilized case, the number of singular values used is $\tilde{r}(T)=r-\#\{\sigma_i:\sigma_1/\sigma_i>10^{15}\}$, where $r$ is the rank of $(\mathbf{B}\mathbf{D})^{T}\mathbf{B}\mathbf{D}$.

\begin{table}[htbp]
\caption{Condition numbers and stability condition values for the stabilized case with $\delta_s = 0.5$, indicated by the subscript, as well as for the unstabilized case.}\label{table:Stability-Condition-Test1}
\begin{center}
  \begin{tblr}{|p{0.3cm}||p{1.7cm}|p{1.9cm}|p{1.1cm}||p{1.7cm}|p{1.9cm}|p{1.1cm}|}\hline
   \centering$\ell$   & \centering$\kappa(\mathbb{A}^{\text{SLOD},\ell}_\mathrm{stab})$& \centering$\kappa(\mathbb{A}^{\text{HSLOD},\ell}_\mathrm{stab})$& \centering$\mu_{\ell,\mathrm{stab}}$ & \centering$\kappa(\mathbb{A}^{\text{SLOD},\ell})$&\centering $\kappa(\mathbb{A}^{\text{HSLOD},\ell})$&\centering $\mu_{\ell}$\\[-0.38ex] \hline
\centering 0  	&	\centering1	& \centering  1 & \centering0&\centering	1	  &	\centering	1	  &\centering\\[-0.7ex] \hline
\centering1  	&	\centering4	& \centering  2 & \centering7.6e-15	& \centering  4	  &	\centering	2	  & \centering7.6e-15\\[-0.7ex] \hline
\centering2  	&	\centering16   & \centering  5 & \centering0.27&\centering	5e+08 &	\centering	5e+07 & \centering1.1\\[-0.7ex] \hline
\centering3  	&\centering	35  & \centering 17 & \centering0.49&	\centering2e+09 &	\centering	3e+08 & \centering11.9\\[-0.7ex] \hline
\centering4  	&\centering	90   & \centering 16& \centering0.5& \centering	2e+09 &	\centering	3e+09 & \centering12.2\\[-0.7ex] \hline
\centering5   &	\centering350  & \centering 26& \centering0.49& \centering	4e+09 &	\centering	6e+09 & \centering19.3\\ [-0.7ex]\hline
  \end{tblr}
\end{center}
\end{table}

\section{Estimate for $\lambda_{\mathrm{min}}(\mathbf{P}^T\mathbf{P})$}\label{Appendix-Eigen-Estimate-for-mesh-indep}
Let $\widetilde{\omega}_{T}:=\widetilde{\omega}_{T}^{(\ell,m)}$, where $\widetilde{\omega}_{T}^{(\ell,m)}$ is defined as in \Cref{Sec-Stable-corrections}. Consider $\hat{\theta}_{\ell,T}^{\text{SLOD}}=\frac{\tilde{\theta}_{\ell,T}^{\text{SLOD}}}{\|\tilde{\theta}_{\ell,T}^{\text{SLOD}}\|_a}$, with 
\begin{equation}
\tilde{\theta}_{\ell,T}^{\text{SLOD}}={\mathcal{E}_{\widetilde{\omega}_{T}}\left(\bar{\psi}^{\text{LOD}}_{\ell,T}+\sum_{K\in \mathcal{T}_{\ell,\widetilde{\omega}_T}\setminus\{T\}}c_K \bar{\psi}_{\ell,K}^{(T)}\right)}.
\end{equation}

Thus, $\hat{\theta}_{\ell,T}^{\text{SLOD}}$ is the energy-norm normalized version of $\tilde{\theta}_{\ell,T}^{\text{SLOD}}$.  Also, let $\tilde{\theta}_{\ell,T}^{\text{LOD}}=\mathcal{E}_{\widetilde{\omega}_{T}}\big(\bar{\psi}^{\text{LOD}}_{\ell,T}\big)$ and $\widetilde{\psi}_{\ell,K}^{(T)}=\mathcal{E}_{\widetilde{\omega}_{T}}\big(\bar{\psi}^{(T)}_{\ell,K}\big)$. Recall that $S_{\omega_{J_i}}:=\{T\in \mathcal{T}_{\ell,\omega_{J_i}}\;:\;\hat{\theta}_{\ell,T}^{\text{SLOD}}=0 \text{ in } \Omega\setminus \omega_{J_i}\}$, with $\omega_{J_i}=\omega_{J_i}^{(\ell,m)}$,  $J_i:=\left\lfloor \frac{i-1} {(2^d-1)^{\min\{\ell,1\}}} \right \rfloor +1$, and $\omega_{J_i}^{(\ell,m)}$ as defined in \Cref{Sec-Construction}. Then, with $\tilde{d}^{(i)}_T=\frac{d_T^{(i)}}{\|\tilde{\theta}_{\ell,T}^{\text{SLOD}}\|_a}$, we have

\begin{eqnarray}
\nonumber \hat{\varphi}^{\text{HSLOD}}_{\ell,i}=\sum_{T\in S_{\omega_{J_i}}}d_T^{(i)} \hat{\theta}_{\ell,T}^{\text{SLOD}}=\sum_{T\in S_{\omega_{J_i}}}\tilde{d}^{(i)}_T \tilde{\theta}_{\ell,T}^{\text{SLOD}}&=&\sum_{T\in S_{\omega_{J_i}}}\tilde{d}^{(i)}_T \left(\tilde{\theta}_{\ell,T}^{\text{LOD}}+ \sum_{K\in\mathcal{T}_{\ell,\widetilde{\omega}_{T}}\setminus\{T\}}  c_{K}^{(T)}\widetilde{\psi}_{\ell,K}^{(T)}\right),
\end{eqnarray}
where $|c_{K}^{(T)}|\leq \delta_s$ for $K\in \mathcal{T}_{\ell,\widetilde{\omega}_T}\setminus \{T\}$ (see condition \eqref{SLOD-stability-condition}). It follows that
\begin{equation}\label{HSLOD-projection-appdx}
\Pi_{\ell}\hat{\varphi}^{\text{HSLOD}}_{\ell,i}=\sum_{T\in S_{\omega_{J_i}}}\tilde{d}^{(i)}_T \left(\Pi_{\ell}\tilde{\theta}_{\ell,T}^{\text{LOD}}+ \sum_{K\in\mathcal{T}_{\ell,\widetilde{\omega}_T}\setminus\{T\}}c_{K}^{(T)}\Pi_{\ell}\widetilde{\psi}_{\ell,K}^{(T)}\right)=\sum_{K\in\mathcal{T}_{\ell,\omega_{J_i}}}\left(\sum_{T\in S_{\omega_{J_i}}}\tilde{d}^{(i)}_T c_K^{(T)}\right)\chi_{K}.
\end{equation}
where $c_{K}^{(T)}=0$ if $K\notin \mathcal{T}_{\ell,\widetilde{\omega}_T}$, $c_{T}^{(T)}=1$, and $|c_{K}^{(T)}|\leq \delta_s$ for $K\in \mathcal{T}_{\ell,\widetilde{\omega}_T}\setminus \{T\}$.

With $\Pi_{\ell-1}\tilde{\theta}_{\ell,T_j}^{\text{SLOD}}=\sum_{T\in \mathcal{T}_{\ell-1,{\omega_{J_i}}}}q_{T}^{(j)}\chi_T$, $N_q=\#\mathcal{T}_{\ell-1,{\omega_{J_i}}}$, and $N_s=\# S_{\omega_{J_i}}$, define the matrix $\mathbf{Q}\in \mathbb{R}^{N_{q} \times N_{s}}$  such that $\mathbf{Q}_{pj}=q^{(j)}_{T_p}$. Let $\mathbf{K}\in \mathbb{R}^{N_{s}\times \# \mathrm{Ker}(\mathbf{Q})}$ be the matrix whose columns form an orthonormal basis of $\mathrm{Ker}(\mathbf{Q})$. Further, with $N_{c}=\# \mathcal{T}_{\ell,\omega_{J_i}}$, define the matrix $\mathbf{C}\in \mathbb{R}^{N_{c}\times N_{s}}$ given by $\mathbf{C}_{pj}=c_{K_p}^{(T_j)}$, where $c_{K_p}^{(T_j)}$ is such that $\Pi_{\ell}\tilde{\theta}_{\ell,T_j}^{\text{SLOD}}=\sum_{K_p\in\mathcal{T}_{\ell,\omega_{J_i}}}c_{K_p}^{(T_j)}\chi_{K_p}$. Then, with $\tilde{\mathbf{d}}^{(i)}=(\tilde{d}^{(i)}_{T})_{T\in S_{\omega_{J_i}}}$, condition \eqref{Projection1-pw-for-Aorthogonality} implies $\tilde{\mathbf{d}}^{(i)}=\mathbf{K}\mathbf{y}$ for some $\mathbf{y}\in \mathbb{R}^{\# \mathrm{Ker}(\mathbf{Q})}$. To satisfy condition \eqref{HSLOD-Stability-Practical-Condition} we need
\begin{equation}\label{HSLOD-conditions-discrete-version}
\mathbf{C}\mathbf{K}\mathbf{y}=\mathbf{e}_{T_i},
\end{equation}
where $\mathbf{e}_{T_i}\in \mathbb{R}^{N_c}$ is the canonical vector whose entries are zero except for the entry associated with the element $T_i\in \mathcal{T}_{\ell}$. In general, $\mathbf{C}\mathbf{K}$ is a rectangular matrix. Consequently, \eqref{HSLOD-conditions-discrete-version} can only be solved in the least-squares-error sense. Then, we have
\begin{equation}\label{d-coeffs-lsqr-solution}
\mathbf{y}=\left((\mathbf{C}\mathbf{K})^T \mathbf{C}\mathbf{K}\right)^{-1}(\mathbf{C}\mathbf{K})^T \mathbf{e}_{T_i}.
\end{equation}
Note that all entries of $\mathbf{C}$, $\mathbf{K}$, and $\mathbf{e}_{T_i}$ are $\mathcal{O}(1)$ (i.e., mesh independent). Furthermore, from \eqref{SLOD-stability-condition}, the absolute value of the entries of $\mathbf{C}$ are bounded by $1$. Then, from \eqref{d-coeffs-lsqr-solution} we have that the entries of $\tilde{\mathbf{d}}^{(i)}$ are $\mathcal{O}(1)$. Also, from \cite[Eq. (2.9)]{FeP20} we have that LOD basis functions are $\mathcal{O}\left(H_{\ell}^{\frac{d}{2}-1}\right)$. Consequently, $\|\tilde{\theta}_{\ell,T}^{\text{SLOD}}\|_a=\mathcal{O}\left(H_{\ell}^{\frac{d}{2}-1}\right)$. Then, since $d_{T}^{(i)}=\tilde{d}^{(i)}_T \|\tilde{\theta}_{\ell,T}^{\text{SLOD}}\|_a$, it follows that $|d_{T}^{(i)}|=\mathcal{O}\left(H_{\ell}^{\frac{d}{2}-1}\right)$. Then, we have
\begin{equation}\label{HSLOD-energy-norm-bound-appdx}
\|\hat{\varphi}^{\text{HSLOD}}_{\ell,i}\|_a=\|\sum_{T\in\mathcal{T}_{\ell}}d_T^{(i)} \hat{\theta}_{\ell,T}^{\text{SLOD}}\|_a\leq \sum_{T\in\mathcal{T}_{\ell}}|d_T^{(i)}| \|\hat{\theta}_{\ell,T}^{\text{SLOD}}\|_a=\sum_{T\in\mathcal{T}_{\ell}}|d_T^{(i)}|\leq C_{\dagger} N_{s} H^{\frac{d}{2}-1}_{\ell}
\end{equation}
where $C_{\dagger}$ is independent of the mesh size but depends on $\beta$ and $\alpha$.
Define
\begin{equation*}
\tilde{p}_{K}^{(i)}\coloneqq \begin{cases}
 \sum_{T\in\mathcal{T}_{\ell}}\tilde{d}^{(i)}_T c_K^{(T)} & \text{if }K\in \mathcal{T}_{\ell,\omega_{J_i}}\\
 0& \text{otherwise}.
 \end{cases}
\end{equation*}
Then, from \eqref{HSLOD-projection-appdx} we have that
\begin{equation*}
\Pi_{\ell}\hat{\varphi}^{\text{HSLOD}}_{\ell,i}=\sum_{K\in \mathcal{T}_{\ell}}\tilde{p}_{K}^{(i)}\chi_{K}.
\end{equation*}
Since $c_{K}^{(T)}$ and $\tilde{d}^{(i)}_{T}$ are $\mathcal{O}(1)$, we also have $\tilde{p}_{K}^{(i)}=\mathcal{O}(1)$.

From the definitions of $\mathbf{P}$, $\widetilde{\mathbf{P}}$, and $\mathbf{N}$ in \Cref{Subsec:HSLOD-Behavior} we have $\mathbf{P}$=$\widetilde{\mathbf{P}}\mathbf{N}$, $\widetilde{\mathbf{P}}_{ij}=\tilde{p}_{K_i}^{(j)}$ with $K_i\in\mathcal{T}_{\ell}$, and $\mathbf{N}_{ii}=\frac{1}{\|\hat{\varphi}^{\text{HSLOD}}_{\ell,i}\|_a}$. Using the Rayleigh quotient, we have for all $ \mathbf{x}\in \mathbb{R}^{N^{b}_{\ell}}\setminus \{\mathbf{0}\}$ and $\mathbf{y}=\mathbf{N}\mathbf{x}$ that
\begin{eqnarray*}
\mathbf{x}^T \mathbf{P}^T  \mathbf{P} \mathbf{x}=\nonumber \mathbf{x}^T \mathbf{N} \widetilde{\mathbf{P}}^T  \widetilde{\mathbf{P}} \mathbf{N} \mathbf{x} &=& \mathbf{y}^T (\widetilde{\mathbf{P}}^T \widetilde{\mathbf{P}}) \mathbf{y}\\ \nonumber
 &\geq& \lambda_{\mathrm{min}}(\widetilde{\mathbf{P}}^T \widetilde{\mathbf{P}})\mathbf{y}^T \mathbf{y}\\ \nonumber
&=& \lambda_{\mathrm{min}}(\widetilde{\mathbf{P}}^T\widetilde{\mathbf{P}}) \mathbf{x}^T \mathbf{N}^2 \mathbf{x}\\
&\geq & \lambda_{\mathrm{min}}(\widetilde{\mathbf{P}}^T\widetilde{\mathbf{P}}) \lambda_{\mathrm{min}}(\mathbf{N}^2) \mathbf{x}^T \mathbf{x}.
\end{eqnarray*}
Consequently, we have 
\begin{equation*}\label{Decoupling-Effects-on-min-eig}
\lambda_{\mathrm{min}}(\mathbf{P}^T\mathbf{P})\geq \lambda_{\mathrm{min}}(\widetilde{\mathbf{P}}^T\widetilde{\mathbf{P}}) \lambda_{\mathrm{min}}(\mathbf{N}^2).
\end{equation*}
From the definition of $\mathbf{N}$ and \eqref{HSLOD-energy-norm-bound-appdx}, we have $\lambda_{\mathrm{min}}^{-1}(\mathbf{N}^2)\leq C_{\dagger}^2 N_{s}^2 H^{d-2}_{\ell}$. Note that $\lambda_{\mathrm{min}}(\widetilde{\mathbf{P}}^T \widetilde{\mathbf{P}})=\mathcal{O}(1)$,
i.e., $\lambda_{\mathrm{min}}(\widetilde{\mathbf{P}}^T \widetilde{\mathbf{P}})$ is mesh-size independent. Therefore, for $\ell\geq 0$, we have
\begin{equation}\label{Eigen-Min-Estimate-Ptilde}
\lambda_{\mathrm{min}}^{-1}(\mathbf{P}^T \mathbf{P})\leq \lambda_{\mathrm{min}}^{-1}(\mathbf{N}^2)^{-1} \lambda_{\mathrm{min}}(\widetilde{\mathbf{P}}^T \widetilde{\mathbf{P}})^{-1}\leq \frac{C_{\dagger}^2N_{s}^2H^{d-2}_{\ell}}{ \lambda_{\mathrm{min}}(\widetilde{\mathbf{P}}^T \widetilde{\mathbf{P}}) },
\end{equation}
where $C_{\dagger}^2N_{s}^2$ and $\lambda_{\mathrm{min}}(\widetilde{\mathbf{P}}^T \widetilde{\mathbf{P}})$ are mesh independent quantities. Since the estimate for $\lambda_{\mathrm{min}}^{-1}(\mathbf{P}^T \mathbf{P})$ depends on $\lambda_{\mathrm{min}}(\widetilde{\mathbf{P}}^T \widetilde{\mathbf{P}})$, we study the behavior of this quantity next.

For $\ell>0$, if conditions \eqref{Projection1-pw-for-Aorthogonality} and \eqref{HSLOD-Stability-Practical-Condition} were satisfied exactly we would have 
\begin{eqnarray*}
\tilde{p}_{T_{i}}^{(i)}&=& \sum_{T\in\mathcal{T}_{\ell}}\tilde{d}^{(i)}_T c_{T_i}^{(T)}=1,\\ \nonumber
\tilde{p}_{\widetilde{T}_{i}}^{(i)}&=& \sum_{T\in\mathcal{T}_{\ell}}\tilde{d}^{(i)}_T c_{\widetilde{T}_{i}}^{(T)}=-1 ,\\ \nonumber
\tilde{p}_{K}^{(i)}&=& \sum_{T\in\mathcal{T}_{\ell}}\tilde{d}^{(i)}_T c_K^{(T)}=0\quad {\text{for all }}K \notin \{T_i,\widetilde{T}_{i}\},
\end{eqnarray*}
for fixed elements $T_i,\widetilde{T}_i \in\mathrm{ref}\big(\omega_{J_i}^{(\ell,0)}\big)$ associated with $\hat{\varphi}^{\text{HSLOD}}_{\ell,i}$.
This implies that $\widetilde{\mathbf{P}^T}\widetilde{\mathbf{P}}$ is a block diagonal matrix, and each of its blocks is a matrix with all diagonal entries equal to $2$ and all off-diagonal entries equal to $1$. Regardless of the size of these blocks, their smallest eigenvalue is always one. Consequently, if both conditions \eqref{Projection1-pw-for-Aorthogonality} and \eqref{HSLOD-Stability-Practical-Condition} are satisfied exactly, we have $\lambda_{\mathrm{min}}(\widetilde{\mathbf{P}}^T \widetilde{\mathbf{P}})=1$. In general, however, the second condition cannot be satisfied exactly but only in the least-squares-error sense. But, if the least-squares error is small enough we should have $\lambda_{\mathrm{min}}(\widetilde{\mathbf{P}}^T \widetilde{\mathbf{P}})$ far enough from $0$.

The basis functions for the case $\ell=0$ are regular SLOD basis functions and thus $\widetilde{\mathbf{P}}=\mathbf{C}$. From condition \eqref{SLOD-stability-condition} we have that $\mathbf{C}=\mathbf{I}+\delta$, were $\mathbf{I}$ is the identity matrix, and $\delta$ is a matrix with zero diagonal entries and all off-diagonal entries less than $\delta_s$ in absolute value. Hence, $\mathbf{C}$ is a matrix obtained from a perturbation of the identity matrix. If $\delta_s$ is small enough, the smallest eigenvalue of $\mathbf{C}^T \mathbf{C}$ should be far enough from $0$, and in that case the basis at level $\ell=0$ will be well-behaved.

From \eqref{Eigen-Min-Estimate-Ptilde} and \eqref{HSLOD-block-condnum-bound} it follows that the condition numbers of the diagonal blocks $\hat{\mathbb{A}}_{\ell \ell}^{H_L}$ of the stiffness matrix $\hat{\mathbb{A}}_{H_L}$ are mesh independent for $\ell>0$ and $\mathcal{O}(H_{0}^{-2})$ for $\ell=0$ (in accordance with \cite[Assumption (2.5)]{HaPe21b}).

\begin{remark}
In case the local orthogonalization procedure is applied, and performed via the QR factorization method, the normalization step contained in that algorithm to make the vectors have a unit euclidean norm will swap the dependence on $H_{\ell}$ of $\lambda_{\mathrm{min}}(\widetilde{\mathbf{P}^T}\widetilde{\mathbf{P}})$ and $\lambda_{\mathrm{min}}(\mathbf{N}^2)$. In that case we still have $\lambda_{\mathrm{min}}^{-1}(\mathbf{P}^T\mathbf{P}) \leq C_{\natural}N_{s}^2H^{d-2}_{\ell}$ for some $C_{\natural}\geq 0$ that is independent of the mesh size, and this implies that the condition number of the level-blocks of $\hat{\mathbb{A}}_{H_L}$ are still mesh independent. 
\end{remark}

\bibliographystyle{abbrv}
\bibliography{bib}

\end{document}